\RequirePackage{fix-cm}
\documentclass[smallextended]{svjour3}       
\smartqed  
\usepackage{graphicx}
\usepackage[backend=biber, maxbibnames=5]{biblatex}

\usepackage{amsfonts} 
\usepackage{mathtools}
\usepackage{hyperref}
\usepackage{amssymb}
\usepackage{tikz}
\usepackage{tikz-cd}
\usepackage[affil-it]{authblk}
\usepackage{blkarray}
\usepackage[english]{babel}
\usepackage[ruled,vlined]{algorithm2e}

\DeclareMathOperator{\Hom}{Hom}

\DeclareMathSymbol{\ast}{\mathbin}{symbols}{"03}

\DeclareMathOperator{\Res}{Res}

\DeclareMathOperator{\Vol}{Vol}

\DeclareMathOperator{\conv}{conv}

\DeclareMathOperator{\rc}{rc}

\DeclareMathOperator{\HF}{HF}

\DeclareMathOperator{\MV}{MV}

\DeclareMathOperator{\trop}{trop}

\newcommand{\cc}[1]{{#1}}
\begin{document}

\title{Mixed subdivisions suitable for the greedy Canny-Emiris formula
}


\author{Carles Checa         \and
        Ioannis Z. Emiris 
}


\institute{C. Checa \at
National \& Kapodistrian University of Athens, and Athena Research Center,
              \email{ccheca@di.uoa.gr}, orcid.org/0000-0003-3477-8876        
           \and
           Ioannis Z. Emiris \at
 Athena Research Center, and National \& Kapodistrian University of Athens,
              \email{emiris@athenarc.gr},  
            orcid.org/0000-0002-2339-5303
}

\date{\today}

\maketitle

\begin{abstract}
The Canny-Emiris formula \cite{cannyemiris93} gives the sparse resultant as the ratio of the determinant of a Sylvester-type matrix over a minor of it, both obtained via a mixed subdivision algorithm.
In \cite{issacversion}, the same authors gave an explicit class of mixed subdivisions 
for the greedy approach so that the formula holds, and 
the dimension of the constructed matrices is smaller
than that of the subdivision algorithm, following the approach of Canny and Pedersen in \cite{cannypedersen}. Our method improves upon the dimensions of the matrices when the Newton polytopes are zonotopes and the systems are multihomogeneous. In this text, we provide more such cases, and we conjecture which might be the liftings providing minimal size of the resultant matrices. We also describe two applications of this formula, namely in computer vision and in the implicitization of surfaces, while offering the corresponding \textit{JULIA} code. We finally introduce a novel tropical approach that leads to an alternative proof of a result in \cite{issacversion}.

\end{abstract}

\section{Introduction}
\label{section:1}

Sparse (or toric) resultants offer today a standard as well as efficient way of studying algebraic systems while exploiting their structure. They have numerous applications in elimination and implicitization theory and many other areas of algebraic geometry. We examine matrix-based methods for expressing and computing this resultant, since they allow us to reduce critical questions to linear algebra computations. 
\medskip

The Canny-Emiris formula was stated in \cite{cannyemiris93} as a rational formula for the sparse resultant that generalizes Macaulay's classic formula in \cite{macaulay}. It gives a combinatorial construction of a Sylvester-type matrix $\mathcal{H}_{\mathcal{A},\rho}$ depending on the family of polynomial supports $\mathcal{A} = (\mathcal{A}_0,\dots,\mathcal{A}_n)$ in a lattice $M$ of rank $n$, and a mixed subdivision $S(\rho)$ defined from a lifting function $\rho$ on the Minkowski sum $\Delta$ of the Newton polytopes $\Delta_i = \conv(\mathcal{A}_i)$. Each row of this matrix corresponds to a lattice point $b \in M$ contained in a translation $\delta$ of the polytope $\Delta$. Moreover, to each lattice point we can associate a type vector $t_b = (t_{b,0},\dots,t_{b,n})$ corresponding to the dimensions of the components $D_i \subset \Delta_i$ of the cell $D \in S(\rho)$ in which $b$ is lying. Imposing that
\[ 
\sum_{i = 0}^nt_{b,i} = n, \quad \forall b \in (\Delta + \delta) \cap M,
\] 
it is possible to build such matrix $\mathcal{H}_{\mathcal{A},\rho}$ and a principal submatrix $\mathcal{E}_{\mathcal{A},\rho}$ of $\mathcal{H}_{\mathcal{A},\rho}$ so that the sparse resultant can be expressed in the form:

\[\Res_{\mathcal{A}} = \frac{\det(\mathcal{H}_{\mathcal{A},\rho})}{\det(\mathcal{E}_{\mathcal{A},\rho})}.\]

D'Andrea, Jerónimo, and Sombra proved this formula in their seminal article \cite{dandrea2020cannyemiris} under the assumption that $S(\rho)$ admits an incremental chain of mixed subdivisions:
\[ 
S(\theta_0) \preceq \dots \preceq S(\theta_n) \preceq S(\rho)\]
satisfying some combinatorial properties, where $\preceq$ denotes that $S(\theta_i)$ refines $S(\theta_{i-1})$, namely each cell $D \in S(\theta_i)$ is contained in a cell of $S(\theta_{i-1})$. This proof extended the first proof given by D'Andrea in \cite{dandrea2002} for generalized unmixed systems. 

\medskip
On the other hand, and quite earlier, Canny and Pedersen gave in \cite{cannypedersen} another approach for the construction of the matrix $\mathcal{H}_{\mathcal{A},\rho}$ besides the subdivision-based algorithm. Starting at a lattice point $b \in M$, one can construct the matrix by only adding the rows corresponding to the columns that have a nonzero entry in a previously considered row. This is a greedy way of understanding these matrices: their construction only considers the strictly necessary rows and columns given the mixed subdivision $S(\rho)$.

\medskip
This paper aims to give a family of lifting functions $\rho$ that correspond to mixed subdivisions for which:
\begin{itemize}
    \item[i)] the proof of the formula in \cite{dandrea2020cannyemiris} holds, and
    \item[ii)] the size of the matrices is reduced, since this determines complexity.
\end{itemize} 
The family that we propose is associated with a vector $v \in \Hom(M_{\mathbb{R}},\mathbb{R})$ outside the hyperplane arrangement associated with the polytope $\Delta$. We extend the proof that verifies that these matrices satisfy the conditions in \cite{dandrea2020cannyemiris} by giving a self-contained result in tropical geometry on refinement of mixed subdivisions; see Theorem \ref{tropicalrefinement}.
\medskip

We expect this family of lifting functions to reduce the size of the Canny-Emiris matrices obtained by the greedy algorithm for a general sparse system. We measure this reduction only for the case where the Newton polytopes are zonotopes generated by $n$ independent line segments. Namely, we simplify the computations to the case where the supports $\mathcal{A}_0,\dots,\mathcal{A}_n$ are:
\[ 
\mathcal{A}_i = \big\{(b_j)_{j = 1,\dots,n} \in \mathbb{Z}^n \quad | \quad 0 \leq b_j \leq a_{ij} \big\}, \quad i = 0,\dots,n ,
\]
assuming that $0 < a_{0j} \leq \dots \leq a_{n-1j}$ for all $j = 1,\dots,n$. The main results in \cite{issacversion} is Theorem \ref{maintheorem} and show that the greedy algorithm will end by reaching only those lattice points with type vector $t_{b}$ satisfying:
\[ \sum_{i = 0}^I t_{b,i} \leq I + 1, \quad \forall I < n.\]

To find all the lattice points in cells with a given type vector, we introduce the type functions:
\[ 
\varphi_b: \{1,\dots,n\} \xrightarrow[]{} \{0,\dots,n\}, \quad t_{b,i} = |\varphi_b^{-1}(i)|.
\]
These combinatorial objects contain all the information of the cells of the mixed subdivision and can help us construct the matrices. In Corollary \ref{bigconclusion}, we give a combinatorial measure of the number of rows of the matrix $\mathcal{H}_{\mathcal{G}}$ given by the greedy algorithm  as:
\[ \sum_{\varphi_b: \{1,\dots,n\} \xrightarrow[]{} \{0,\dots,n\}}\prod_{j = 1}^na_{\varphi_b(j)j}\]
where $\varphi_b$ satisfies $|\varphi^{-1}_b(\{0,\dots,I\})| \leq I + 1$. This result is not optimal amongst all the possible mixed subdivsions of an $n$-zonotope system, but it might be so amongst the ones given by affine lifting functions; see Examples \ref{smallexample}, \ref{exampledrop}, \ref{bigexample}.

\medskip
In \cite{issacversion}, we also showed that some multihomogeneous resultant matrices can be seen as an instance of the previous case by embedding their Newton polytopes and the mixed subdivisions of their Minkowski sum into an $n$-zonotope. Despite the existence of many exact determinantal formulas for some of these cases; see \cite{bender2021koszultype,mixedgrobnerbasistowards,cannyemirisincremental,emimantzdeterminant,sturmfelszelevinsky}, we expect our approach to have an easier generalization to general sparse systems through the use of the type functions and the underlying combinatorics. On the other, we beleive that using non-affine lifting functions, one could possibly find the minimal matrices from one which one can have the Canny-Emiris formula for the sparse resultant. In Conjecture \ref{conjecture}, we state a conjecture on the greedy subsets involved in such minimal construction for multi-homogeneous systems.
\medskip

\color{black}
We motivate the practical use of resultant matrices by showing its applications in computer vision and computer-aided geometric design. Namely, we show how resultant matrices apply in the solution of the $5$-points problem \cite{Bhayani2019ASR} and to the surface implicitization problem \cite{clementlaurent,LAURENT201414,kalinkaetal}. In both cases, the goal of reducing the size of the matrices appears to be one of the main lines for possible improvement. We add examples of these formulations in the \href{https://github.com/carleschecanualart/CannyEmiris}{\color{blue}JULIA implementation \color{gray}} of the treated cases ($n$-zonotopes as Newton polytopes, and multihomogeneous systems). More than improving the existing formulas (which give, in general, smaller Sylvester matrices), the goal of this implementation is to introduce the type functions, which are instrumental in the construction of the resultant matrices.

\medskip
This paper is based on the results of the authors' paper~\cite{issacversion}. Compared to that version, here we provide a new family of examples, a conjecture on the minimal size of resultant matrices, as well as two applications of resultant formulas in computer vision and in implicitization. We also improve the method of proof of Theorem~\ref{admissiblechain} through a different approach which uses tropical geometry.
\medskip

This paper is structured as follows.
In Section \ref{section:1}, we summarize the proof of the Canny-Emiris formula in \cite{dandrea2020cannyemiris} and we explain the greedy approach of \cite{cannypedersen}. In Section \ref{section:2}, a concrete family of mixed subdivisions is given by considering a hyperplane arrangement associated to the Newton polytopes, thus formalizing an important aspect of the combinatorial construction. It is then proven that the Canny-Emiris formula holds when the construction uses this arrangement. \color{black} In Section~\ref{tropicalrefinement}, we prove a result on tropical geometry, providing an alternative proof of Theorem \ref{admissiblechain}. Section~\ref{section:3} is devoted to combinatorially finding the size of the Canny-Emiris matrices when the Newton polytopes are zonotopes generated by $n$ line segments. \color{black} In Section~\ref{sec::4}, we conjecture that the liftings providing minimal size of the matrices are related to the degree reverse lexicographical monomial order. In Section \ref{sec::5}, we have added the examples of application of the Canny-Emiris formula in the $5$-point problem and in surface implicitization. 

\subsection{The Canny-Emiris formula}

Let $M$ be a lattice of rank $n$ and $M_{\mathbb{R}} = M \otimes \mathbb{R}$ the corresponding real vector space. Let $N = \Hom(M,\mathbb{Z})$ be its dual and $\mathbb{T}_N = N \otimes \mathbb{C}^{\times}$ the underlying torus. Let $\mathcal{A}_0,\dots,\mathcal{A}_n \subset M$ be a family of supports corresponding to the polynomials:
\[ 
F_i = \sum_{a \in \mathcal{A}_i}u_{i,a}\chi^a \in \mathbb{Z}[u_{i,a}][M], \quad a \in \mathcal{A}_i \quad i = 0,\dots,n,
\]
where $\chi^a$ are the characters in $\mathbb{T}_N$ of the lattice points $a \in \mathcal{A}_i$. Let $\Delta_i = \conv(\mathcal{A}_i) \subset M_{\mathbb{R}}$ for $i = 0,\dots,n$ be the convex hulls of the supports, also known as Newton polytopes, and $\Delta$ their Minkowski sum in $M_{ \mathbb{R}}$. \\

The incidence variety $Z(\mathbf{F})$ is defined as the zero set in $\mathbb{T}_N \times \prod_{i = 0}^n\mathbb{P}^{\mathcal{A}_i}$ of the polynomials $\mathbf{F} = (F_0,\dots,F_n)$. Denote by $\pi: \mathbb{T}_N \times \prod_{i = 0}^n\mathbb{P}^{\mathcal{A}_i} \xrightarrow[]{} \prod_{i = 0}^n\mathbb{P}^{\mathcal{A}_i}$ the projection onto the second factor and let $\pi_*(Z(\mathbf{F}))$ be the direct image of the zero set of $F_0,\dots,F_n$.

\begin{definition}
The sparse resultant, denoted as $\Res_{\mathcal{A}}$, is any primitive polynomial in $\mathbb{Z}[u_{i,a}]$ defining the direct image $\pi_{*}(Z(\mathbf{F}))$. 
\end{definition}

There are some lattice operations that can help us simplify the computation of these objects.

\begin{lemma}
\cite[Proposition 3.2]{dandrea2020cannyemiris}  Let $\phi: M \xrightarrow[]{} M'$ be a monomorphism of lattices of rank $n$. Then, $\Res_{\phi(\mathcal{A})} = \Res_{\mathcal{A}}^{[M': \phi(M)]}$. 
\label{lemmamonomrphism}
\end{lemma}

\begin{remark}
\label{zero}
Moreover, the sparse resultant is invariant under translations. Therefore, we can always assume $0 \in \mathcal{A}_i$ for all $i = 0,\dots,n$.
\end{remark}

\begin{definition}
\label{infconvolution}
A mixed subdivision of $\Delta$ is a decomposition of this polytope in a union of cells $\Delta = \cup D$ such that:
\begin{itemize}
    \item[i)] the intersection of two cells is either a cell or empty,
    \item[ii)] every face of a cell is also a cell of the subdivision and,
    \item[iii)]  every cell $D$ has a component structure $D = D_0 + \dots + D_n$ where $D_i$ is a cell of the subdivision in $\Delta_i$.
\end{itemize}

\medskip
The usual way to construct mixed subdivisions is by considering piecewise affine convex lifting functions $\rho_i: \Delta_i \xrightarrow[]{} \mathbb{R}$ as explained in \cite{gkz1994}. A global lifting function $\rho: \Delta \xrightarrow[]{} \mathbb{R}$ is obtained after taking the inf-convolution of the previous functions, as explained in \cite[Sec. 2]{dandrea2020cannyemiris}.
\end{definition}

\begin{definition}
A mixed subdivision of $\Delta$ is tight if, for every $n$-cell $D$, its components satisfy:
\[ \sum_{i=0}^n\dim D_i = n.\]
\end{definition}
In the case of $n + 1$ polynomials and $n$ variables, this property guarantees that every $n$-cell has a component that is $0$-dimensional. The cells that have a single $0$-dimensional component are called mixed ($i$-mixed if it is the $i$-th component). The rest of the cells are called non-mixed. \\

Let $\delta$ be a generic vector such that the lattice points in the interior of $\Delta + \delta$ lie in $n$-cells. Then, consider:
 \[ \mathcal{B} = (\Delta + \delta) \cap M.\]
Each element $b \in \mathcal{B}$ lies in one of these translated cells $D + \delta$ and let $D_i$ be the components of this cell. As the subdivision is tight, there is at least one $i$ such that $\dim D_i = 0$.\\

Following the language of \cite{michielscools}, we call $t_b = (t_{b,0},\dots,t_{b,n})$ the type vector associated with $b$, defined as $t_{b,i} = \dim D_i$ for $b \in D + \delta$.

 \begin{definition}
 The row content is a function 
 \[ 
 \rc: \mathcal{B} \xrightarrow[]{} \cup_{i = 0}^n\{i\} \times \mathcal{A}_i \] where, for $b \in \mathcal{B}$ lying in an  $n$-cell $D$, $\rc(b)$ is a pair $(i(b),a(b))$ with $i(b) = \max\{i \in \{0,\dots,n\} \: | \: t_{b,i} = 0\}$ and $a(b) = D_{i(b)}$.
 \end{definition}
This provides a partition of $\mathcal{B}$ into subsets:
 \[ \mathcal{B}_i = \{b \in \mathcal{B} \quad | \quad i(b) = i\}.
 \]
 Finally, we construct the Canny-Emiris matrices $\mathcal{H}_{\mathcal{A},\rho}$ whose rows correspond to the coefficients of the polynomials $\chi^{b - a(b)}F_{i(b)}$ for each of the $b \in \mathcal{B}$. In particular, the entry corresponding to a pair $b,b' \in \mathcal{B}$ is:
\[\mathcal{H}_{\mathcal{A},\rho}[b,b'] =  \begin{cases} 
      u_{i(b),b'-b+a(b)} & b' - b + a(b) \in \mathcal{A}_i \\
      0 & \text{otherwise}
   \end{cases}
\]

\begin{remark}
Each entry contains, at most, a single coefficient $u_{i,a}$. In particular, the row content allows us to choose a maximal submatrix of $\mathcal{H}_{\mathcal{A},\rho}$ from the matrix of a map sending a tuple of polynomials $(G_0,\dots,G_n)$ to $G_0F_0 + \dots + G_nF_n$. 
\end{remark}

Let $\mathcal{C} \subset \mathcal{B}$ be a subset of the supports in translated cells. The matrix $\mathcal{H}_{\mathcal{A},\rho,\mathcal{C}}$ is defined by considering the submatrix of the corresponding rows and columns associated with elements in $\mathcal{C}$. In particular, we look at the set of lattice points lying in translated non-mixed cells and consider:
\[ \mathcal{B}^{\circ} = \{b \in \mathcal{B} \: | \:  b \text{ lies in a translated non-mixed cell}\}.\]
With this, we form the principal submatrix:
\[\mathcal{E}_{\mathcal{A},\rho} = \mathcal{H}_{\mathcal{A},\rho,\mathcal{B}^{\circ}}\]
The Canny-Emiris conjecture states that the sparse resultant is the quotient of the determinants of these two matrices:
\[\Res_{\mathcal{A}} = \frac{\det(\mathcal{H}_{\mathcal{A},\rho})}{\det(\mathcal{E}_{\mathcal{A},\rho})}.\]
This result was conjectured by Canny and Emiris and proved by D'Andrea, Jerónimo, and Sombra under the restriction that the mixed subdivision $S(\rho)$ given by the lifting $\rho$ satisfies a certain condition, given on a chain of mixed subdivisions.

\begin{definition}
Let $S(\phi), S(\psi)$ be two mixed subdivisions of $\Delta = \sum_{n = 0}^n\Delta_i$. We say that $S(\psi)$ refines $S(\phi)$ and write $S(\phi) \preceq S(\psi)$ if for every cell $C \in S(\psi)$ there is a cell $D \in S(\phi)$ such that $C \subset D$. An incremental chain of mixed subdivisions $S(\theta_0) \preceq \dots \preceq S(\theta_n)$ is a chain of mixed subdivisions of $\Delta$ refining each other.

\end{definition}
\begin{remark}
\label{remarkzerolifting}
In \cite[Definition 2.4]{dandrea2020cannyemiris}, a common lifting function $\omega \in \prod_{i = 0}^n\mathbb{R}^{\mathcal{A}_i}$ is considered and the $S(\theta_i)$ are given by the lifting functions $\omega^{<i} = (\omega_0,\dots,\omega_{i-1},0)$ as long as $S(\theta_i) \preceq S(\theta_{i+1})$. The last zero represents the lifting on $(\Delta_i,\dots,\Delta_n)$. The resulting mixed subdivision is the same as if we considered the zero lifting in $\sum_{j = i}^n\Delta_j$.
\end{remark}
\begin{definition}
The mixed volume of $n$ polytopes $P_1,\dots,P_n \subset M_{\mathbb{R}}$, denoted as $\MV_{M}(P_1,\dots,P_n)$, is the coefficient of $\prod_{i = 1}^n\lambda_i$ in:
\[ \Vol_n(\lambda_1 P_1 + \dots + \lambda_nP_n)\]
which is a polynomial in $\lambda_1,\dots,\lambda_n$ \cite[Theorem 6.7]{coxlitosh}.
\end{definition}
\begin{proposition} \cite[Theorem 3.4]{emirisrege}
\label{mixedvolume1}
Let $S(\rho)$ be a tight mixed subdivision of $\Delta = (\Delta_0,\dots,\Delta_n)$. For $i = 0,\dots,n$, the mixed volume of all the polytopes except $\Delta_i$ equals the volume of the $i$-mixed cells.
\[\MV(\Delta_0,\dots,\Delta_{i-1},\Delta_{i+1},\dots,\Delta_n) = \sum_{D \: i\text{-mixed}}\Vol_nD\]
\end{proposition}
In particular, $\MV(\Delta_0,\dots,\Delta_{i-1},\Delta_{i+1},\dots,\Delta_n)$ equals the degree of the sparse resultant in the coefficients of $F_i$; see \cite[Chapter 7, Theorem 6.3]{coxlitosh}. Each of the rows of $\mathcal{H}_{\mathcal{A},\rho}$ will correspond to a lattice point $b$ and each entry on that row will have degree $1$ with respect to the coefficients of $F_{i(b)}$ and zero with respect to the coefficients of the rest of polynomials. Therefore, if we add the lattice points in $i$-mixed cells, the degree of $\mathcal{H}_{\mathcal{A},\rho}$ with respect to the coefficients of $F_i$ will be at least the degree of the resultant with respect to the same coefficients. 
\begin{definition}\cite[Definition 3.4]{dandrea2020cannyemiris}
The fundamental subfamily of $\mathcal{A}$ is the minimal family of supports $\mathcal{A}_I = (\mathcal{A}_i)_{i \in I}$ such that the resultant has positive degree with respect to the coefficients of $F_i$ for $i \in I$. This definition can be given in other equivalent terms as shown in \cite[Corollary 1.1]{sturmfels94}.
\end{definition}
\begin{remark}
\label{remarkresult}
Using Proposition \ref{mixedvolume1}, we can see that if the fundamental subfamily is empty, then the resultant is equal to $1$ while if the fundamental subfamily is $\{i\}$ then $\mathcal{A}_i$ is given by a single point $\{a\}$ and the resultant is $u_{i,a}^{m_i}$ for $m_i = \MV(\Delta_0,\dots,\Delta_{i-1},\Delta_{i+1},\dots,\Delta_n)$. The Canny-Emiris formula holds naturally \cite[Proposition 4.26]{dandrea2020cannyemiris} in both cases. \end{remark}
\begin{definition}
\label{admissible}
An incremental chain $S(\theta_0) \preceq \dots \preceq S(\theta_n)$ is admissible if for each $i = 0,\dots,n$, each $n$-cell $D$ of the subdivision $S(\theta_i)$ satisfies either of the following two conditions
\begin{itemize}
    \item[i)] the fundamental subfamily of $\mathcal{A}_D$ contains at most one support or
    \item[ii)] $\mathcal{B}_{D,i}$ is contained in the union of the translated $i$-mixed cells of $S(\rho_D)$.
\end{itemize}
A mixed subdivision $S(\rho)$ is called admissible if it admits an admissible incremental chain $S(\theta_0) \preceq \dots \preceq S(\theta_n) \preceq S(\rho)$ refining it.
\end{definition}

With all these properties, together with the use of the product formulas, one can reproduce the proof of the Canny-Emiris formula given in \cite[Theorem 4.27]{dandrea2020cannyemiris} under the conditions of admissibility in $S(\rho)$; see also \cite{poissondansom,sturmfels94}.

\subsection{The greedy algorithm}

Using the previous notation, we state the greedy algorithm in \cite{cannypedersen} for the construction of the matrix. Let $b \in \mathcal{B}$ be a lattice point in a translated cell. The first step of the algorithm is to add the row of the matrix corresponding to $b$, and then continue by considering the lattice points corresponding to the columns that have a nonzero entry in this row. These lattice points are:
\[ b - a(b) + \mathcal{A}_{i(b)}.\]
All these lattice points will have to be added as rows of the matrix. If we add the lattice point $b'$ at some point of the algorithm after having added another lattice point $b$, we say that we \textit{reach} $b'$ from $b$. The algorithm terminates when there are no more lattice points to add and it might give a square matrix $\mathcal{H}_{\mathcal{G}}$ which has less rows and columns than $\mathcal{H}_{\mathcal{A},\rho}$, which was  constructed using all the lattice points in $\mathcal{B}$. The rows and columns associated to lattice points in non-mixed cells also provide a minor $\mathcal{E}_{\mathcal{G}}$ of $\mathcal{H}_{\mathcal{G}}$.
\medskip

It was not proved by Canny and Pedersen whether this approach would always include all the lattice points in mixed cells as rows of the matrix, independently of the starting point. As these points are necessary to achieve the degree of the resultant, we consider them to be the starting points of the algorithm.
\begin{remark}
\label{diagonal}
We know that the entry corresponding to the diagonal of the matrix $\mathcal{H}_{\mathcal{A},\rho,\mathcal{C}}$ will be $\prod_{b \in \mathcal{C}}u_{i(b),a(b)}$ for any subset $\mathcal{C} \subset \mathcal{B}$. This term can be used in order to deduce that these matrices have non-zero determinant; see \cite[Proposition 4.13]{dandrea2020cannyemiris}.
\end{remark}
\begin{theorem}
If the Canny-Emiris formula holds for a mixed subdivision $S(\rho)$ and the greedy algorithm provides matrices $\mathcal{H}_{\mathcal{G}}$ and $\mathcal{E}_{\mathcal{G}}$ by starting at the lattice points in mixed cells, then:
\[ \Res_{\mathcal{A}} = \frac{ \det(\mathcal{H}_{\mathcal{G}})}{\det(\mathcal{E}_{\mathcal{G}})}.\]
\end{theorem}
\begin{proof}
In general, there is a subset $\mathcal{G} \subset \mathcal{B}$ corresponding to the rows and columns of $\mathcal{H}_{\mathcal{G}}$. We are assuming that $\mathcal{G}$ contains all the lattice points in translated mixed cells. Let $\mathcal{H}_{\mathcal{A},\rho}$ be the matrix containing all lattice points in translated cells of $S(\rho)$. Without loss of generality, we can assume that the matrix takes the following form:
\[
\mathcal{H}_{\mathcal{A},\rho} = \begin{pmatrix}
  \mathcal{H}_{\mathcal{G}} & 0 \\
  \bullet & \mathcal{H}_{\mathcal{B} - \mathcal{G}}
\end{pmatrix}
\]
where $\mathcal{H}_{\mathcal{G}}$ is the minor corresponding to the lattice points in $\mathcal{G}$ and $\mathcal{H}_{\mathcal{B} - \mathcal{G}}$ is the minor corresponding to the lattice points not in $\mathcal{G}$. The zeros appear due to the fact that there is no pair $b \notin \mathcal{G}$, $b' \in \mathcal{G}$ such that $b \in b' - a(b') + \mathcal{A}_{i(b)}$. The same block-triangular structure also appears in the principal submatrix $\mathcal{E}_{\mathcal{A},\rho}$ and all the lattice points that are not in $\mathcal{G}$ must be non-mixed, implying that $\mathcal{E}_{\mathcal{B} - \mathcal{G}} = \mathcal{H}_{\mathcal{B} - \mathcal{G}}$. 
\medskip

Finally, using the fact that the determinant of a block-triangular matrix is the product of the determinants of the diagonal blocks, we can prove the resultant formula:
\[ \Res_{\mathcal{A}} = \frac{\det(\mathcal{H}_{\mathcal{A},\rho})}{\det(\mathcal{E}_{\mathcal{A},\rho})} = \frac{\det(\mathcal{H}_{\mathcal{G}})·\det(\mathcal{H}_{\mathcal{B}-\mathcal{G}})}{\det(\mathcal{E}_{\mathcal{G}})·\det(\mathcal{H}_{\mathcal{B}-\mathcal{G}})} = \frac{\det(\mathcal{H}_{\mathcal{G}})}{\det(\mathcal{E}_{\mathcal{G}})}.\]
\end{proof}
\begin{example}
\label{smallexample}
Let $f_0,f_1,f_2$ be three bilinear equations corresponding to the supports $\mathcal{A}_0 =  \mathcal{A}_1 =  \mathcal{A}_2 = \{(0,0),(1,0),(0,1),(1,1)\}$. A possible mixed subdivision $S(\rho)$ is the following:
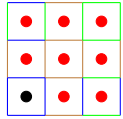
\[\begin{tikzpicture}

\draw[brown] (1/2,1/2) -- (1,1/2);
\draw[brown] (1/2,1/2) -- (1/2,1);
\draw[brown] (1,1/2) -- (1,1);
\draw[brown] (1/2,1) -- (1,1);
\draw[brown] (1/2,0) -- (1,0);
\draw[brown] (0,1/2) -- (0,1);
\draw[brown] (1/2,3/2) -- (1,3/2);
\draw[brown] (3/2,1/2) -- (3/2,1);

\draw[blue] (0,0) -- (1/2,0);
\draw[blue] (0,0) -- (0,1/2);
\draw[blue] (1/2,0) -- (1/2,1/2);
\draw[blue] (0,1/2) -- (1/2,1/2);
\draw[blue] (1,1/2) -- (1,0);
\draw[blue] (1/2,1) -- (0,1);
\draw[blue] (3/2,1/2) -- (3/2,0);
\draw[blue] (1/2,3/2) -- (0,3/2);

\draw[green] (1,1) -- (3/2,1);
\draw[green] (1,1) -- (1,3/2);
\draw[green] (3/2,3/2) -- (3/2,1);
\draw[green] (3/2,3/2) -- (1,3/2);
\draw[green] (3/2,0) -- (1,0);
\draw[green] (0,3/2) -- (0,1);
\draw[green] (3/2,1/2) -- (1,1/2);
\draw[green] (1/2,3/2) -- (1/2,1);

\filldraw[black] (0+1/4,1/4) circle (2pt) node[anchor=east] {};
\filldraw[red] (0+1/4,1/2+1/4) circle (2pt) node[anchor=east] {};
\filldraw[red] (0+1/4,1+1/4) circle (2pt) node[anchor=east] {};
\filldraw[red] (1/2+1/4,0+1/4) circle (2pt) node[anchor=south ] {};
\filldraw[red] (1/2+1/4,1/2+1/4) circle (2pt) node[anchor= south] {};
\filldraw[red] (1/2+1/4,1+1/4) circle (2pt) node[anchor=west] {};
\filldraw[red] (1+1/4,0+1/4) circle (2pt) node[anchor=south] {};
\filldraw[red] (1+1/4,1/2+1/4) circle (2pt) node[anchor=west] {};
\filldraw[red] (1+1/4,1+1/4) circle (2pt) node[anchor=west] {};

\end{tikzpicture}\]
where the dots indicate the lattice points in translated mixed cells. The number of lattice points in translated cells is $9$. However, if we construct the matrix greedily starting from the lattice points in translated mixed cells, we have an $8 \times 8$ matrix, corresponding to the lattice points marked in red.
\end{example}

\color{black}
\begin{example}
\label{exampledrop}
Let $f_0,f_1,f_2$ be three bihomogeneous equations with supports $$\mathcal{A}_0 = \{(0,0),(1,0),(2,0),(0,1),(1,1),(2,1)\}, $$$$ \mathcal{A}_1 = \{(0,0),(1,0),(0,1),(1,1),(0,2),(1,2)\}, \: \mathcal{A}_2 = \{(0,0),(1,0),(0,1),(0,1)\}.$$The expected number of supports lying in translated cells is $16$. Let $\rho_0 = (0,3,6,3,6,9)$, $\rho_1 = (0,2,2,4,4,6)$ and $\rho_2 = (0,1,1,2)$ be the lifting functions and $\delta = (-\frac{1}{2},\frac{1}{2})$ give the following mixed subdivision:
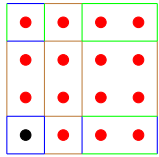
\[\begin{tikzpicture}

\draw[brown] (1/2,1/2) -- (1/2,3/2);
\draw[brown] (1/2,1/2) -- (1,1/2);
\draw[brown] (1/2,3/2) -- (1,3/2);
\draw[brown] (1,1/2) -- (1,3/2);
\draw[brown] (0,1/2) -- (0,3/2);
\draw[brown] (2,1/2) -- (2,3/2);
\draw[brown] (1/2,0) -- (1,0);
\draw[brown] (1/2,2) -- (1,2);

\draw[blue] (0,0) -- (1/2,0);
\draw[blue] (0,0) -- (0,1/2);
\draw[blue] (1/2,0) -- (1/2,1/2);
\draw[blue] (0,1/2) -- (1/2,1/2);
\draw[blue] (1,0) -- (1,1/2);
\draw[blue] (2,0) -- (2,1/2);
\draw[blue] (0,3/2) -- (1/2,3/2);
\draw[blue] (0,2) -- (1/2,2);

\draw[green] (1,3/2) -- (2,3/2);
\draw[green] (1,3/2) -- (1,2);
\draw[green] (2,2) -- (2,3/2);
\draw[green] (2,2) -- (1,2);
\draw[green] (1,0) -- (2,0);
\draw[green] (0,3/2) -- (0,2);
\draw[green] (1,1/2) -- (2,1/2);
\draw[green] (1/2,3/2) -- (1/2,2);

\filldraw[black] (0+1/4,1/4) circle (2pt) node[anchor=east] {};
\filldraw[red] (0+1/4,1/2+1/4) circle (2pt) node[anchor=east] {};
\filldraw[red] (0+1/4,1+1/4) circle (2pt) node[anchor=east] {};
\filldraw[red] (0+1/4,3/2+1/4) circle (2pt) node[anchor=east] {};
\filldraw[red] (1/2+1/4,0+1/4) circle (2pt) node[anchor=south ] {};
\filldraw[red] (1/2+1/4,1/2+1/4) circle (2pt) node[anchor= south] {};
\filldraw[red] (1/2+1/4,1+1/4) circle (2pt) node[anchor=west] {};
\filldraw[red] (1/2+1/4,3/2+1/4) circle (2pt) node[anchor=west] {};
\filldraw[red] (1+1/4,0+1/4) circle (2pt) node[anchor=south] {};
\filldraw[red] (1+1/4,1/2+1/4) circle (2pt) node[anchor=west] {};
\filldraw[red] (1+1/4,1+1/4) circle (2pt) node[anchor=west] {};
\filldraw[red] (1+1/4,3/2+1/4) circle (2pt) node[anchor=west] {};
\filldraw[red] (3/2+1/4,0+1/4) circle (2pt) node[anchor=south] {};
\filldraw[red] (3/2+1/4,1/2+1/4) circle (2pt) node[anchor=west] {};
\filldraw[red] (3/2+1/4,1+1/4) circle (2pt) node[anchor=west] {};
\filldraw[red] (3/2+1/4,3/2+1/4) circle (2pt) node[anchor=west] {};

\end{tikzpicture}.\]

However, if we use the greedy approach, we have an $15 \times 15$ matrix, corresponding to the lattice points marked in red.


%
%

\end{example}
\begin{example}
\label{bigexample}
Let $f_0,f_1,f_2,f_3$ be four polynomials with $$\mathcal{A}_0 = \mathcal{A}_1 = \mathcal{A}_2 = \mathcal{A}_3 = \{(0,0,0),(1,0,0),(0,1,0),(0,0,1),(1,0,1),(0,1,1)\}$$ and $\rho_0 = (0,3,6,3,6,9)$, $\rho_1 = (0,2,4,2,4,6)$, $\rho_2 = (0,1,2,1,2,3)$ and $\rho_3 = (0,0,0,0,0,0)$ gives the mixed subdivision:
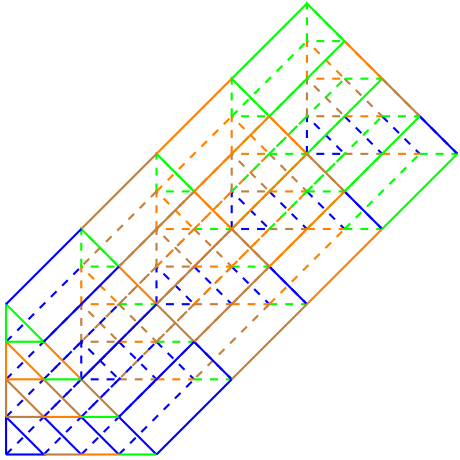
\[
\begin{tikzpicture}[thick,scale=5]

\coordinate (A1) at (0,0);
\coordinate (A2) at (0,0.1);
\coordinate (A3) at (0,0.2);
\coordinate (A4) at (0,0.3);
\coordinate (A5) at (0,0.4);
\coordinate (A6) at (0.1,0);
\coordinate (A7) at (0.1,0.1);
\coordinate (A8) at (0.1,0.2);
\coordinate (A9) at (0.1,0.3);
\coordinate (A10) at (0.2,0);
\coordinate (A11) at (0.2,0.1);
\coordinate (A12) at (0.2,0.2);
\coordinate (A13) at (0.3,0);
\coordinate (A14) at (0.3,0.1);
\coordinate (A15) at (0.4,0);

\coordinate (B1) at (0.2,0.2);
\coordinate (B2) at (0.2,0.3);
\coordinate (B3) at (0.2,0.4);
\coordinate (B4) at (0.2,0.5);
\coordinate (B5) at (0.2,0.6);
\coordinate (B6) at (0.3,0.2);
\coordinate (B7) at (0.3,0.3);
\coordinate (B8) at (0.3,0.4);
\coordinate (B9) at (0.3,0.5);
\coordinate (B10) at (0.4,0.2);
\coordinate (B11) at (0.4,0.3);
\coordinate (B12) at (0.4,0.4);
\coordinate (B13) at (0.5,0.2);
\coordinate (B14) at (0.5,0.3);
\coordinate (B15) at (0.6,0.2);

\coordinate (C1) at (0.4,0.4);
\coordinate (C2) at (0.4,0.5);
\coordinate (C3) at (0.4,0.6);
\coordinate (C4) at (0.4,0.7);
\coordinate (C5) at (0.4,0.8);
\coordinate (C6) at (0.5,0.4);
\coordinate (C7) at (0.5,0.5);
\coordinate (C8) at (0.5,0.6);
\coordinate (C9) at (0.5,0.7);
\coordinate (C10) at (0.6,0.4);
\coordinate (C11) at (0.6,0.5);
\coordinate (C12) at (0.6,0.6);
\coordinate (C13) at (0.7,0.4);
\coordinate (C14) at (0.7,0.5);
\coordinate (C15) at (0.8,0.4);

\coordinate (D1) at (0.6,0.6);
\coordinate (D2) at (0.6,0.7);
\coordinate (D3) at (0.6,0.8);
\coordinate (D4) at (0.6,0.9);
\coordinate (D5) at (0.6,1.0);
\coordinate (D6) at (0.7,0.6);
\coordinate (D7) at (0.7,0.7);
\coordinate (D8) at (0.7,0.8);
\coordinate (D9) at (0.7,0.9);
\coordinate (D10) at (0.8,0.6);
\coordinate (D11) at (0.8,0.7);
\coordinate (D12) at (0.8,0.8);
\coordinate (D13) at (0.9,0.6);
\coordinate (D14) at (0.9,0.7);
\coordinate (D15) at (1.0,0.6);

\coordinate (E1) at (0.8,0.8);
\coordinate (E2) at (0.8,0.9);
\coordinate (E3) at (0.8,1.0);
\coordinate (E4) at (0.8,1.1);
\coordinate (E5) at (0.8,1.2);
\coordinate (E6) at (0.9,0.8);
\coordinate (E7) at (0.9,0.9);
\coordinate (E8) at (0.9,1.0);
\coordinate (E9) at (0.9,1.1);
\coordinate (E10) at (1.0,0.8);
\coordinate (E11) at (1.0,0.9);
\coordinate (E12) at (1.0,1.0);
\coordinate (E13) at (1.1,0.8);
\coordinate (E14) at (1.1,0.9);
\coordinate (E15) at (1.2,0.8);

\begin{scope}[thick,dashed,,opacity=0.6]

\draw[blue] (B1) -- (B2) -- (B6) -- (B1);
\draw[blue] (B7) -- (B10);
\draw[blue] (B11) -- (B13);
\draw[blue] (B14) -- (B15);

\draw[brown] (B2) -- (B3) -- (B7) -- (B2);
\draw[brown] (B6) -- (B10);
\draw[brown] (B8) -- (B11);

\draw[orange] (B3) -- (B4) -- (B8) -- (B3);
\draw[orange] (B7) -- (B11);
\draw[orange] (B10) -- (B13);

\draw[green] (B4) -- (B5) -- (B9) -- (B4);
\draw[green] (B8) -- (B12);
\draw[green] (B11) -- (B14);
\draw[green] (B13) -- (B15);

\draw[blue] (C1) -- (C2) -- (C6) -- (C1);
\draw[blue] (C7) -- (C10);
\draw[blue] (C11) -- (C13);
\draw[blue] (C14) -- (C15);

\draw[brown] (C2) -- (C3) -- (C7) -- (C2);
\draw[brown] (C6) -- (C10);
\draw[brown] (C8) -- (C11);
\draw[brown] (C12) -- (C14);

\draw[orange] (C3) -- (C4) -- (C8) -- (C3);
\draw[orange] (C9) -- (C12);
\draw[orange] (C7) -- (C11);
\draw[orange] (C10) -- (C13);

\draw[green] (C4) -- (C5) -- (C9) -- (C4);
\draw[green] (C8) -- (C12);
\draw[green] (C11) -- (C14);
\draw[green] (C13) -- (C15);

\draw[blue] (D1) -- (D2) -- (D6) -- (D1);
\draw[blue] (D7) -- (D10);
\draw[blue] (D11) -- (D13);
\draw[blue] (D14) -- (D15);

\draw[brown] (D2) -- (D3) -- (D7) -- (D2);
\draw[brown] (D6) -- (D10);
\draw[brown] (D8) -- (D11);
\draw[brown] (D12) -- (D14);

\draw[orange] (D3) -- (D4) -- (D8) -- (D3);
\draw[orange] (D9) -- (D12);
\draw[orange] (D7) -- (D11);
\draw[orange] (D10) -- (D13);

\draw[green] (D4) -- (D5) -- (D9) -- (D4);
\draw[green] (D8) -- (D12);
\draw[green] (D11) -- (D14);
\draw[green] (D13) -- (D15);

\draw[blue] (E1) -- (E2) -- (E6) -- (E1);
\draw[blue] (E7) -- (E10);
\draw[blue] (E11) -- (E13);
\draw[blue] (E14) -- (E15);

\draw[brown] (E2) -- (E3) -- (E7) -- (E2);
\draw[brown] (E6) -- (E10);
\draw[brown] (E8) -- (E11);
\draw[brown] (E12) -- (E14);

\draw[orange] (E3) -- (E4) -- (E8) -- (E3);
\draw[orange] (E9) -- (E12);
\draw[orange] (E7) -- (E11);
\draw[orange] (E10) -- (E13);

\draw[green] (E4) -- (E5) -- (E9) -- (E4);
\draw[green] (E8) -- (E12);
\draw[green] (E11) -- (E14);
\draw[green] (E13) -- (E15);

\draw[brown] (C1) -- (B1);
\draw[brown] (C2) -- (B2);
\draw[brown] (C3) -- (B3);
\draw[brown] (C4) -- (B4);
\draw[brown] (C6) -- (B6);
\draw[brown] (C7) -- (B7);
\draw[brown] (C8) -- (B8);
\draw[brown] (C10) -- (B10);
\draw[brown] (C11) -- (B11);
\draw[brown] (C13) -- (B13);

\draw[orange] (C1) -- (D1);
\draw[orange] (C2) -- (D2);
\draw[orange] (C3) -- (D3);
\draw[orange] (C4) -- (D4);
\draw[orange] (C6) -- (D6);
\draw[orange] (C7) -- (D7);
\draw[orange] (C8) -- (D8);
\draw[orange] (C10) -- (D10);
\draw[orange] (C11) -- (D11);
\draw[orange] (C13) -- (D13);

\draw[green] (D1) -- (E1);
\draw[green] (D2) -- (E2);
\draw[green] (D3) -- (E3);
\draw[green] (D4) -- (E4);
\draw[green] (D6) -- (E6);
\draw[green] (D7) -- (E7);
\draw[green] (D8) -- (E8);
\draw[green] (D10) -- (E10);
\draw[green] (D11) -- (E11);
\draw[green] (D13) -- (E13);

\draw[blue] (A1) -- (B1);
\draw[blue] (A2) -- (B2);
\draw[blue] (A3) -- (B3);
\draw[blue] (A4) -- (B4);
\draw[blue] (A6) -- (B6);
\draw[blue] (A7) -- (B7);
\draw[blue] (A8) -- (B8);
\draw[blue] (A10) -- (B10);
\draw[blue] (A11) -- (B11);
\draw[blue] (A13) -- (B13);

\end{scope}

\draw[blue] (A1) -- (A2) -- (A6) -- (A1);
\draw[blue] (A7) -- (A10);
\draw[blue] (A11) -- (A13);
\draw[blue] (A14) -- (A15);

\draw[brown] (A2) -- (A3) -- (A7) -- (A2);
\draw[brown] (A6) -- (A10);
\draw[brown] (A8) -- (A11);
\draw[brown] (A12) -- (A14);

\draw[orange] (A3) -- (A4) -- (A8) -- (A3);
\draw[orange] (A9) -- (A12);
\draw[orange] (A7) -- (A11);
\draw[orange] (A10) -- (A13);

\draw[green] (A4) -- (A5) -- (A9) -- (A4);
\draw[green] (A8) -- (A12);
\draw[green] (A11) -- (A14);
\draw[green] (A13) -- (A15);

\draw[blue] (B14) -- (B15);
\draw[blue] (C14) -- (C15);
\draw[blue] (D14) -- (D15);
\draw[blue] (E14) -- (E15);

\draw[brown] (B12) -- (B14);
\draw[brown] (C12) -- (C14);
\draw[brown] (D12) -- (D14);
\draw[brown] (E12) -- (E14);

\draw[orange] (B9) -- (B12);
\draw[orange] (C9) -- (C12);
\draw[orange] (D9) -- (D12);
\draw[orange] (E9) -- (E12);

\draw[green] (B5) -- (B9);
\draw[green] (C5) -- (C9);
\draw[green] (D5) -- (D9);
\draw[green] (E5) -- (E9);

\draw[green] (E5) -- (D5);
\draw[green] (E9) -- (D9);
\draw[green] (E12) -- (D12);
\draw[green] (E14) -- (D14);
\draw[green] (E15) -- (D15);

\draw[orange] (C5) -- (D5);
\draw[orange] (C9) -- (D9);
\draw[orange] (C12) -- (D12);
\draw[orange] (C14) -- (D14);
\draw[orange] (C15) -- (D15);

\draw[brown] (C5) -- (B5);
\draw[brown] (C9) -- (B9);
\draw[brown] (C12) -- (B12);
\draw[brown] (C14) -- (B14);
\draw[brown] (C15) -- (B15);

\draw[blue] (A5) -- (B5);
\draw[blue] (A9) -- (B9);
\draw[blue] (A12) -- (B12);
\draw[blue] (A14) -- (B14);
\draw[blue] (A15) -- (B15);

\end{tikzpicture}\]
If we take the translation $\delta = (-2/3,-2/3,-1/2)$ the number of points in traslated mixed cells is $24$, but the degree of the resultant is $3 + 3 + 3 + 3 = 12$. If we start at the point $(0,0,0)$ and use the greedy algorithm, we achieve a matrix of size $20 \times 20$.
\end{example}

\section{A family of mixed subdivisions}
\label{section:2}

In this section, we give a family of lifting functions associated to the polytopes $\Delta_0,\dots,\Delta_n$ and we prove that the Canny-Emiris formula holds for the corresponding mixed subdivisions. 

\begin{definition}
We can define a hyperplane arrangement $\mathcal{H} \subset N_{\mathbb{R}}$ by considering the span of the $(n-1)$-dimensional cones of the normal fan of $\Delta$; see \cite{ziegler} for more on polytopes and hyperplane arrangements.
\end{definition}
\begin{example}
A polytope $\Delta$ (green), together with its normal fan (blue) and the hyperplane arrangement $\mathbb{H}_{\Delta}$ (red).
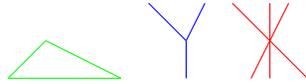
\[\begin{tikzpicture}

\draw[green] (0,1/2) -- (3/2,1/2);
\draw[green] (3/2,1/2) -- (1/2,1);
\draw[green] (0,1/2) -- (1/2,1);

\end{tikzpicture} \quad \begin{tikzpicture}

\draw[blue] (0,0) -- (-1/2,1/2);
\draw[blue] (0,0) -- (1/4,1/2);
\draw[blue] (0,0) -- (0,-1/2);

\end{tikzpicture} \quad \begin{tikzpicture}

\draw[red] (1/2,-1/2) -- (-1/2,1/2);
\draw[red] (-1/4,-1/2) -- (1/4,1/2);
\draw[red] (0,1/2) -- (0,-1/2);

\end{tikzpicture}\]

\end{example}
\begin{definition}
\label{definitionlifting}
Let $\mathcal{H}$ be the hyperplane arrangement associated to $\Delta$ and take a vector $v \in N_\mathbb{R}$ which does not lie in $\mathcal{H}$. We consider lifting functions $\omega_i: \mathcal{A}_i \xrightarrow[]{} \mathbb{R}$ defined as:
\[ \omega_i(x) = \lambda_i\langle v, x\rangle \quad i = 0,\dots,n \quad x \in \Delta_i\]
for $\lambda_0,\dots,\lambda_n \in \mathbb{R}$ satisfying $\lambda_0 > \dots > \lambda_n \geq 0$ and small enough. Let $\rho = (\omega_0,\dots,\omega_n)$ be a family of lifting functions providing the mixed subdivision $S(\rho)$.
\end{definition}

\begin{remark}
\label{remarkdandrea2002} 
This choice of the lifting function can also be seen as a case of the approach of \cite{dandrea2002}, in a first proof of the rational formula for generalized unmixed systems. In particular, it is possible to think of the choice of the row content $a(b)$ associated to each lattice point as trying to solve the simplex method with the lifting function as objective. This family guarantees that we are always choosing this point in the same direction; see Figure \ref{figure1} for a description of this process.
\end{remark}

\begin{theorem}
\label{admissiblechain}
$S(\rho)$ is an admissible mixed subdivision.
\end{theorem}
Proving that $S(\rho)$ is an admissible mixed subdivision consists on both proving that it has an incremental chain satisfying $S(\theta_0) \preceq \dots \preceq S(\theta_n) \preceq S(\rho)$ and that this incremental chain satisfies the conditions in Definition \ref{admissible}.
\medskip

The easiest way to prove the chain condition would be to use 
\cite[Proposition 2.11]{dandrea2020cannyemiris}, which claims that for each $i = 0,\dots,n$, there is an open neighboorhood of $0 \in U \subset \mathbb{R}^{\mathcal{A}_i}$ such that for $\omega_i \in U$ we have $S(\theta_i) \preceq S(\theta_{i+1})$. In this case, for $\lambda_{i+1}$ small enough satisfying $\lambda_i > \lambda_{i+1} > 0$, $\omega_i$ lies in $U$. Therefore, the $S(\theta_i)$ form an incremental chain.
\medskip

\color{black}

We can drop the restriction that $\lambda_{i+1}$ is small enough by proving a more general result. In the following chapter, we explore this new proof in the more general context of tropical geometry. 

\section{Tropical refinement: an extended proof of Theorem \ref{admissiblechain}}
\label{subsection:trop}

This section can be read independently with respect to the rest of sections of this article. We describe, in much broader generality than we need, the refinement of mixed subdivisions. In particular, we draw the full picture of when a coherent mixed subdivision refines another one, by only changing the lifting function in $\Delta_i$. In terms of the previous notation, we would like to know whether $S(\theta_i) \preceq S(\theta_{i+1})$ for some $i = 0,\dots,n$. Instead of studying a given mixed subdivision, we define a dual of such object by introducing tropical geometry. After proving such result using tropical geometry, the family of lifting functions given in Section \ref{section:2} will satisfy the refinement. 
\medskip

\cc{On the other hand, the study of the sparse resultant and the Canny-Emiris formula in the tropical setting is a topic of interest on its own; see \cite{Jensen_2013}. Thus, we expect the results of this section to be interesting also from the point of view of computing the tropical resultant and verifying the validity of the Canny-Emiris formula in this setting.}

\begin{remark}
As in this paper we are mainly interested in affine lifting functions, we restrict to such case. However, the following results could be reproduced for any piecewise affine lifting function.
\end{remark}

The general context of tropical geometry consists of working over rings of polynomials over $\mathbb{R}$ with the tropical operations:
\[ x \oplus y = \min(x,y) \quad x \otimes y = x + y\]

\begin{definition}
A tropical polynomial is the expression:
\[ \trop(f)(x) = \oplus_{a \in \mathcal{A}} \omega_a x^{\otimes a} = \min_{a \in \mathcal{A}}(\omega_a + ax)\]
for $x \in \mathbb{R}^n$
where $\mathcal{A}$ is the support of $f$. A tropical hypersurface $V(\trop(f)))$ in $\mathbb{R}^n$ is the set of points where the previous minimum is attained, at least, twice.
\end{definition}
\begin{remark}
We can consider the coefficients $\omega_{a}$ to be the values of a lifting function. If the lifting is affine, we have $\omega_a = \langle v, a \rangle$ for some vector $v \in N_{\mathbb{R}}$. Therefore, the tropical polynomial with coefficients $\omega_a$ would be:
\[ \min_{a \in \mathcal{A}}(a(x + v))\]
\end{remark}
\begin{definition}
A tropical system $\mathcal{T}_r$ is formed by $r + 1$ tropical polynomials with supports $P_0,\dots,P_{r} \subset M$:
\[\trop(f_i^{\omega_i})(x) = \bigoplus_{a_i \in P_i}\omega_{i,a} \otimes x^{\otimes a} = \min_{a \in P_i}\big(\omega_{i,a} + a \cdot x\big)\]
where the coefficients of the system are given by some lifting function of the $P_i$. In some references like \cite{MaclaganSturmfels}, it is important to specify a valuation in the field but here we can suppose it to be trivial.
\end{definition}
In our context (as in Remark \ref{remarkzerolifting}), we have a family of tropical systems $\mathcal{T}_i$ for $i = 0,\dots,n$ of the supports: \[\mathcal{A}_0,\dots,\mathcal{A}_{i-1},\sum_{j = i}^n\mathcal{A}_j \subset M \]The last tropical polynomial is formed by imposing $0$ coefficients, therefore, it is defined by:
\[ \min_{a \in \sum_{j = r}^n\Delta_j}\langle a,x\rangle\]
which corresponds to the normal fan of $\sum_{j = r}^n\Delta_j$. This coincides with the assumptions for $S(\theta_i)$ in Remark \ref{remarkzerolifting}.
\begin{proposition}
The expression $\min_{a \in \Delta}\langle a,x\rangle$ is achieved twice in the $(n-1)$-dimensional cones of the normal fan of $\Delta = \conv(\mathcal{A})$.
\end{proposition}
\begin{proof}
A $j$-th dimensional cone $\mathcal{N}_F \subset M_{\mathbb{R}}$ of the normal fan of $\Delta$ corresponds to a $(n-j)$-dimensional face of $\Delta$. Take $v \in \mathcal{N}_F$, then the value $\min_{a \in F}\langle a,v\rangle$ is the same for all $a \in F$, which is a face. Therefore, the minimum $\min_{a \in F}\langle a,x\rangle$ is achieved, at least,  twice in $\mathcal{N}_F$ if $j < n$. On the other hand, if the minimum is achieved at least twice at $v$, the convex hull \[\conv\{a \in \Delta \quad \min \langle a, v\rangle \text{ is achieved}\}\] is a positive dimensional face $F$ of $\Delta$, therefore $v$ lies in a cone of dimension at most $(n-1)$ in the normal fan of $\Delta$.
\end{proof}
\begin{proposition}
The expression $\min_{a \in \mathcal{A}}\langle a,x + v\rangle$ is achieved twice in the $(n-1)$-dimensional cones of the normal fan of $\Delta$ translated after $v \in N_{\mathbb{R}}$. 
\end{proposition}

\begin{proof}
The same proof as the previous works after translating by $v$.
\end{proof}
In this context, we can see the tropical system $\mathcal{T}_i$ as the superposition in $N_{\mathbb{R}}$ of the normal fans $\mathcal{F}_0,\dots,\mathcal{F}_n$ centered at different points $v_i \in N_{\mathbb{R}}$ which correspond to each of the lifting functions $\omega_i: \Delta_i \xrightarrow[]{} \mathbb{R}$.

\begin{definition}
A polyhedral complex $\mathcal{P}$ is a union of cells (bounded or unbounded) in $N_{\mathbb{R}}$ such that:
\begin{itemize}
    \item Every face of a cell in $\mathcal{P}$ is also in $\mathcal{P}$.
    \item The (possibly empty) intersection of two cells in $\mathcal{P}$ is also in $\mathcal{P}$.
\end{itemize}
\end{definition}
Fans \cc{(and also the complexes associated to tropical systems)} are a good example of polyhedral complexes. \cc{In the next theorem, we explain the duality between polyhedral complexes and mixed subdivisions, i.e., the $j$-dimensional cells of a polyhedral complex correspond to the $(n-j)$-dimensional cells of the mixed subdivision; see \cite{MaclaganSturmfels}.}
\begin{proposition}
Let $\mathcal{A}_0,\dots,\mathcal{A}_n$ be a family of supports and $\omega: \sum_{i = 0}^n\mathcal{A}_i \xrightarrow[]{} \mathbb{R}$
be a lifting function. The polyhedral complex defined by tropical system $\mathcal{T}$ taking the values of $\omega$ as coefficients is dual to the mixed subdivision $S(\omega)$.
\end{proposition}

\begin{proof}
Let $p$ be a $0$-dimensional cell of the polyhedral complex defined by $\mathcal{T}$. As it is the intersection of cones of each of the fans $\mathcal{F}_i$, there is a cell of $S(\rho)$ corresponding to the sum of the faces associated to \cc{the cones of }each of the fans. On the other hand, an $n$-cell $D$ on the mixed subdivision corresponds to a point $p$, which is the intersection of the normal cones of each of the summands $D_i$. Each of the faces of $D$ corresponds to a cell of the polyhedral complex in which $p$ is contained.
\end{proof}

Let's denote by $\mathbb{H}_i$, the hyperplane arrangement in $\mathbb{R}^n$ associated to the tropical system $\mathcal{T}_i$. Before stating the main theorem, we \cc{add} an example of the refining construction.
\begin{example}
Let $\mathcal{A}_0 = \{(0,0),(1,0),(0,1),(1,1)\}$, $\mathcal{A}_1 = \mathcal{A}_2 = \{(0,0),(1,0),(0,1)\}$ with corresponding convex hulls $\Delta_0,\Delta_1,\Delta_2$. \cc{We s}tart with the trivial mixed subdivision.
    
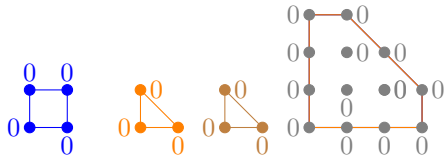
\[        \begin{tikzpicture}

\draw [blue](0,0) -- (1/2,0);
\draw [blue](0,0) -- (0,1/2);
\draw [blue](0,1/2) -- (1/2,1/2);
\draw [blue](1/2,0) -- (1/2,1/2);

\filldraw[blue] (0,0) circle (2pt) node[anchor=east] {0};
\filldraw[blue] (1/2,0) circle (2pt) node[anchor=north] {0};
\filldraw[blue] (0,1/2) circle (2pt) node[anchor=south] {0};
\filldraw[blue] (1/2,1/2) circle (2pt) node[anchor=south] {0};

\end{tikzpicture} \quad
        \begin{tikzpicture}

\draw [orange](0,0) -- (1/2,0);
\draw [orange](0,0) -- (0,1/2);
\draw [orange](0,1/2) -- (1/2,0);

\filldraw[orange] (0,0) circle (2pt) node[anchor=east] {0};
\filldraw[orange] (1/2,0) circle (2pt) node[anchor=north] {0};
\filldraw[orange] (0,1/2) circle (2pt) node[anchor=west] {0};

\end{tikzpicture}
        \begin{tikzpicture}

\draw [brown](0,0) -- (1/2,0);
\draw [brown](0,0) -- (0,1/2);
\draw [brown](0,1/2) -- (1/2,0);
\
\filldraw[brown] (0,0) circle (2pt) node[anchor=east] {0};
\filldraw[brown] (1/2,0) circle (2pt) node[anchor=north] {0};
\filldraw[brown] (0,1/2) circle (2pt) node[anchor=west] {0};

\end{tikzpicture} 
\begin{tikzpicture}

\draw[brown] (0,0) -- (3/2,0);
\draw[brown] (3/2,0) -- (3/2,1/2);
\draw[brown] (0,0) -- (0,3/2);
\draw[brown] (1/2,3/2) -- (0,3/2);
\draw[brown] (1/2,3/2) -- (3/2,1/2);
\draw[blue] (0,0) -- (3/2,0);
\draw[blue] (3/2,0) -- (3/2,1/2);
\draw[blue] (0,0) -- (0,3/2);
\draw[blue] (1/2,3/2) -- (0,3/2);
\draw[blue] (1/2,3/2) -- (3/2,1/2);
\draw[orange] (0,0) -- (3/2,0);
\draw[orange] (3/2,0) -- (3/2,1/2);
\draw[orange] (0,0) -- (0,3/2);
\draw[orange] (1/2,3/2) -- (0,3/2);
\draw[orange] (1/2,3/2) -- (3/2,1/2);

\filldraw[gray] (0,0) circle (2pt) node[anchor=east] {0};
\filldraw[gray] (0,1/2) circle (2pt) node[anchor=east] {0};
\filldraw[gray] (0,1) circle (2pt) node[anchor=east] {0};
\filldraw[gray] (0,3/2) circle (2pt) node[anchor=east] {0};
\filldraw[gray] (1/2,0) circle (2pt) node[anchor=north ] {0};
\filldraw[gray] (1/2,1/2) circle (2pt) node[anchor= north] {0};
\filldraw[gray] (1/2,1) circle (2pt) node[anchor=west] {0};
\filldraw[gray] (1,0) circle (2pt) node[anchor=north] {0};
\filldraw[gray] (1,1/2) circle (2pt) node[anchor=west] {0};
\filldraw[gray] (3/2,0) circle (2pt) node[anchor=north] {0};
\filldraw[gray] (1,1/2) circle (2pt) node[anchor=west] {0};
\filldraw[gray] (1/2,3/2) circle (2pt) node[anchor=west] {0};
\filldraw[gray] (3/2,1/2) circle (2pt) node[anchor=west] {0};
\filldraw[gray] (1,1) circle (2pt) node[anchor=west] {0};

\end{tikzpicture}
\]In this case, the corresponding tropical system is given by the inner normal fan to the Minkowski sum, which corresponds to the superposition of the normal fans of each summand.
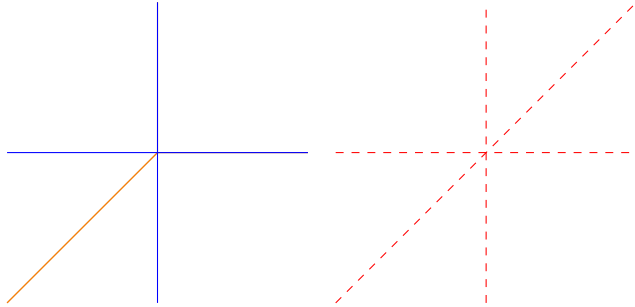
\[
         \begin{tikzpicture}

        \draw [brown](0,0) -- (0,2);
        \draw [brown](0,0) -- (2,0);
        \draw [brown](0,0) -- (-2,-2);
        \draw [orange](0,0) -- (0,2);
        \draw [orange](0,0) -- (2,0);
        \draw [orange](0,0) -- (-2,-2);
        \draw [blue](0,0) -- (0,2);
        \draw [blue] (0,0) -- (2,0);
        \draw [blue] (0,0) -- (0,-2);
        \draw [blue] (0,0) -- (-2,0);

    \end{tikzpicture}
    \quad
         \begin{tikzpicture}
        \draw [dashed,red](-2,0) -- (2,0);
        \draw [dashed,red](-2,-2) -- (2,2);
        \draw [dashed,red](0,-2) -- (0,2);

    \end{tikzpicture}
    \]

The dashed drawing represents the central hyperplane arrangement which we will denote as $\mathbb{H}_0$. Any lifting of $\Delta_0$ will refine the subdivision. However, we can see that refinement corresponds to moving the point $(0,0)$ of the blue fan to an adjacent chamber \cc{in} $\mathbb{H}_0$. Let's take $(2,2)$ as a normal vector. This means lifting $\Delta_0$ after an affine function of type $c - 2x - 2y$. We can choose any constant $c$ (say $c = 4$) as it will give the same lifting. The subdivision looks like:
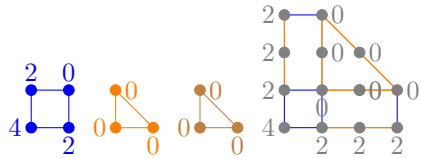
\[   
        \begin{tikzpicture}

\draw [blue](0,0) -- (1/2,0);
\draw [blue](0,0) -- (0,1/2);
\draw [blue](0,1/2) -- (1/2,1/2);
\draw [blue](1/2,0) -- (1/2,1/2);

\filldraw[blue] (0,0) circle (2pt) node[anchor=east] {4};
\filldraw[blue] (1/2,0) circle (2pt) node[anchor=north] {2};
\filldraw[blue] (0,1/2) circle (2pt) node[anchor=south] {2};
\filldraw[blue] (1/2,1/2) circle (2pt) node[anchor=south] {0};

\end{tikzpicture}
        \begin{tikzpicture}

\draw [orange](0,0) -- (1/2,0);
\draw [orange](0,0) -- (0,1/2);
\draw [orange](0,1/2) -- (1/2,0);

\filldraw[orange] (0,0) circle (2pt) node[anchor=east] {0};
\filldraw[orange] (1/2,0) circle (2pt) node[anchor=north] {0};
\filldraw[orange] (0,1/2) circle (2pt) node[anchor=west] {0};

\end{tikzpicture}
        \begin{tikzpicture}

\draw [brown](0,0) -- (1/2,0);
\draw [brown](0,0) -- (0,1/2);
\draw [brown](0,1/2) -- (1/2,0);
\
\filldraw[brown] (0,0) circle (2pt) node[anchor=east] {0};
\filldraw[brown] (1/2,0) circle (2pt) node[anchor=north] {0};
\filldraw[brown] (0,1/2) circle (2pt) node[anchor=west] {0};

\end{tikzpicture}
    \begin{tikzpicture}

\draw[blue] (0,0) -- (1/2,0);
\draw[blue] (1/2,0) -- (1/2,1/2);
\draw[blue] (0,0) -- (0,1/2);
\draw[blue] (0,1/2) -- (1/2,1/2);
\draw[blue] (0,3/2) -- (1/2,3/2);
\draw[blue] (3/2,0) -- (3/2,1/2);
\draw[brown] (1/2,1/2) -- (3/2,1/2);
\draw[brown] (1/2,3/2) -- (3/2,1/2);
\draw[brown] (1/2,0) -- (3/2,0);
\draw[brown] (1/2,3/2) -- (1/2,1/2);
\draw[brown] (0,3/2) -- (0,1/2);
\draw[orange] (1/2,1/2) -- (3/2,1/2);
\draw[orange] (1/2,3/2) -- (3/2,1/2);
\draw[orange] (1/2,0) -- (3/2,0);
\draw[orange] (1/2,3/2) -- (1/2,1/2);
\draw[orange] (0,3/2) -- (0,1/2);

\filldraw[gray] (0,0) circle (2pt) node[anchor=east] {4};
\filldraw[gray] (0,1/2) circle (2pt) node[anchor=east] {2};
\filldraw[gray] (0,1) circle (2pt) node[anchor=east] {2};
\filldraw[gray] (0,3/2) circle (2pt) node[anchor=east] {2};
\filldraw[gray] (1/2,0) circle (2pt) node[anchor=north ] {2};
\filldraw[gray] (1/2,1/2) circle (2pt) node[anchor= north] {0};
\filldraw[gray] (1/2,1) circle (2pt) node[anchor=west] {0};
\filldraw[gray] (1,0) circle (2pt) node[anchor=north] {2};
\filldraw[gray] (1,1/2) circle (2pt) node[anchor=west] {0};
\filldraw[gray] (3/2,0) circle (2pt) node[anchor=north] {2};
\filldraw[gray] (1,1/2) circle (2pt) node[anchor=west] {0};
\filldraw[gray] (1/2,3/2) circle (2pt) node[anchor=west] {0};
\filldraw[gray] (3/2,1/2) circle (2pt) node[anchor=west] {0};
\filldraw[gray] (1,1) circle (2pt) node[anchor=west] {0};

\end{tikzpicture}.\]Moreover, the corresponding tropical system $\mathcal{T}_1$ and the corresponding (not central) hyperplane arrangement $\mathbb{H}_1$ are:
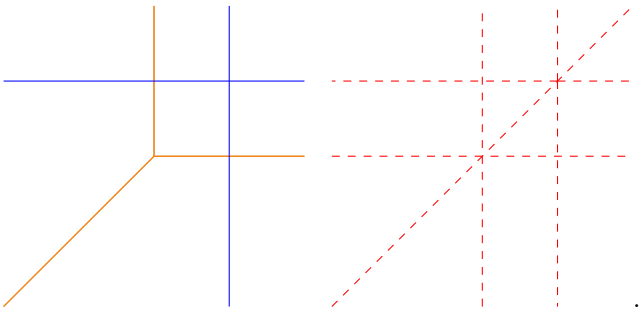
\[
         \begin{tikzpicture}

        \draw [brown](0,0) -- (0,2);
        \draw [brown](0,0) -- (2,0);
        \draw [brown](0,0) -- (-2,-2);
        \draw [orange](0,0) -- (0,2);
        \draw [orange](0,0) -- (2,0);
        \draw [orange](0,0) -- (-2,-2);
        \draw [blue](1,1) -- (1,2);
        \draw [blue] (1,1) -- (2,1);
        \draw [blue] (1,1) -- (1,-2);
        \draw [blue] (1,1) -- (-2,1);

    \end{tikzpicture}\quad 
         \begin{tikzpicture}

        \draw [dashed,red](0,-2) -- (0,2);
        \draw [dashed,red](-2,0) -- (2,0);
        \draw [dashed,red](-2,-2) -- (2,2);
        \draw [dashed,red](1,1) -- (1,2);
        \draw [dashed,red] (1,1) -- (2,1);
        \draw [dashed,red] (1,1) -- (1,-2);
        \draw [dashed,red] (1,1) -- (-2,1);

    \end{tikzpicture}.
    \]
We claim that we can move the orange fan and we will be refining the mixed subdivision. In particular, moving the orange fan to each of the adjacent cells on the hyperplane arrangement corresponds to all the possible ways to refine the previous mixed subdivision. For instance, if we take the translation given by the vector $(1,-1)$, which would be the normal vector to the affine lifting $c - x + y$ with $c = 1$. The mixed subdivision looks like:
        
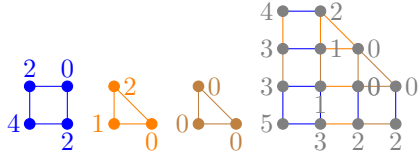
\[   
        \begin{tikzpicture}

\draw [blue](0,0) -- (1/2,0);
\draw [blue](0,0) -- (0,1/2);
\draw [blue](0,1/2) -- (1/2,1/2);
\draw [blue](1/2,0) -- (1/2,1/2);

\filldraw[blue] (0,0) circle (2pt) node[anchor=east] {4};
\filldraw[blue] (1/2,0) circle (2pt) node[anchor=north] {2};
\filldraw[blue] (0,1/2) circle (2pt) node[anchor=south] {2};
\filldraw[blue] (1/2,1/2) circle (2pt) node[anchor=south] {0};

\end{tikzpicture}
        \begin{tikzpicture}

\draw [orange](0,0) -- (1/2,0);
\draw [orange](0,0) -- (0,1/2);
\draw [orange](0,1/2) -- (1/2,0);

\filldraw[orange] (0,0) circle (2pt) node[anchor=east] {1};
\filldraw[orange] (1/2,0) circle (2pt) node[anchor=north] {0};
\filldraw[orange] (0,1/2) circle (2pt) node[anchor=west] {2};

\end{tikzpicture}
        \begin{tikzpicture}

\draw [brown](0,0) -- (1/2,0);
\draw [brown](0,0) -- (0,1/2);
\draw [brown](0,1/2) -- (1/2,0);
\
\filldraw[brown] (0,0) circle (2pt) node[anchor=east] {0};
\filldraw[brown] (1/2,0) circle (2pt) node[anchor=north] {0};
\filldraw[brown] (0,1/2) circle (2pt) node[anchor=west] {0};

\end{tikzpicture}
    \begin{tikzpicture}

\draw[blue] (0,0) -- (1/2,0);
\draw[blue] (1/2,0) -- (1/2,1/2);
\draw[blue] (0,0) -- (0,1/2);
\draw[blue] (0,1/2) -- (1/2,1/2);
\draw[blue] (0,1) -- (1/2,1);
\draw[blue] (0,3/2) -- (1/2,3/2);
\draw[blue] (1,0) -- (1,1/2);
\draw[blue] (3/2,0) -- (3/2,1/2);
\draw[brown] (1,1/2) -- (1,1);
\draw[brown] (1,1/2) -- (3/2,1/2);
\draw[brown] (1,1) -- (3/2,1/2);
\draw[brown] (1,0) -- (3/2,0);
\draw[brown] (1/2,1) -- (1/2,1/2);
\draw[brown] (0,1) -- (0,1/2);
\draw[orange] (1/2,1) -- (1,1);
\draw[orange] (1/2,1) -- (1/2,3/2);
\draw[orange] (1,1) -- (1/2,3/2);
\draw[orange] (0,1) -- (0,3/2);
\draw[orange] (1,1/2) -- (1/2,1/2);
\draw[orange] (1,0) -- (1/2,0);

\filldraw[gray] (0,0) circle (2pt) node[anchor=east] {5};
\filldraw[gray] (0,1/2) circle (2pt) node[anchor=east] {3};
\filldraw[gray] (0,1) circle (2pt) node[anchor=east] {3};
\filldraw[gray] (0,3/2) circle (2pt) node[anchor=east] {4};
\filldraw[gray] (1/2,0) circle (2pt) node[anchor=north ] {3};
\filldraw[gray] (1/2,1/2) circle (2pt) node[anchor= north] {1};
\filldraw[gray] (1/2,1) circle (2pt) node[anchor=west] {1};
\filldraw[gray] (1,0) circle (2pt) node[anchor=north] {2};
\filldraw[gray] (1,1/2) circle (2pt) node[anchor=west] {0};
\filldraw[gray] (3/2,0) circle (2pt) node[anchor=north] {2};
\filldraw[gray] (1,1/2) circle (2pt) node[anchor=west] {0};
\filldraw[gray] (1/2,3/2) circle (2pt) node[anchor=west] {2};
\filldraw[gray] (3/2,1/2) circle (2pt) node[anchor=west] {0};
\filldraw[gray] (1,1) circle (2pt) node[anchor=west] {0};
\end{tikzpicture}
\]
and the tropical system after the translation vector $(1,-1)$, corresponds to:
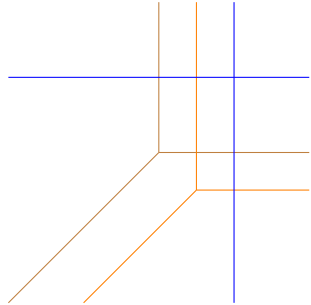
\[
         \begin{tikzpicture}
        \draw [brown](0,0) -- (0,2);
        \draw [brown](0,0) -- (2,0);
        \draw [brown](0,0) -- (-2,-2);
        \draw [orange](1/2,-1/2) -- (1/2,2);
        \draw [orange](1/2,-1/2) -- (2,-1/2);
        \draw [orange](1/2,-1/2) -- (-1,-2);
        \draw [blue](1,1) -- (1,2);
        \draw [blue] (1,1) -- (2,1);
        \draw [blue] (1,1) -- (1,-2);
        \draw [blue] (1,1) -- (-2,1);
    \end{tikzpicture}
    \]
\end{example}
We now recapitulate the notation used so far. Let $\omega_i: \mathcal{A}_i \xrightarrow[]{} \mathbb{R}$ be the lifting function. As in Theorem \ref{admissiblechain}, $S(\theta_i)$ be the mixed subdivisions of the candidate incremental chain given by the lifting functions $(\omega_0,\dots,\omega_{i-1},0,\dots,0)$. Let $\mathcal{T}_i$ be the tropical systems dual to each of the mixed subdivisions $S(\theta_i)$ for $i = 0,\dots,n$. Let $\mathbb{H}_i$ be the hyperplane arrangement associated to each of the tropical systems. At this point, we have all the ingredients to state and prove the tropical refinement.
\begin{theorem}(Tropical refinement) \label{tropicalrefinement} Let $i = 1,\dots,n$. The mixed subdivision $S(\theta_i)$ refines $S(\theta_{i-1})$, if and only if, the normal vector to the lifting function $\omega_{i-1}: \mathcal{A} \xrightarrow[]{} \mathbb{R}$ lives in a chamber of $\mathbb{H}_i$ adjacent to $0 \in \mathbb{R}^n$.
\end{theorem}
We now construct the tools needed for proving this result.
\begin{definition}
We say that a ray $r$ of the normal fan $\mathcal{F}_i$ preserves adjacencies if it is adjacent to the same cells in $\mathcal{T}_i$ and $\mathcal{T}_{i-1}$.
\end{definition}

\begin{lemma}
Let $S(\theta_i)$ be a mixed subdivision of $\Delta_0,\dots,\Delta_{i-1},\sum_{j = i}^n \Delta_j$ for $i = 0,\dots,n$. The lifting of $\Delta_i$ will give $S(\theta_i) \preceq S(\theta_{i+1})$, if and only if, each ray of $\mathcal{F}_i$ preserves the adjacencies after the translation. 
\end{lemma}
\begin{proof}
Suppose there is a ray $r$ that doesn't preserve an adjacencies. Then, take the $0$-dimensional cell of the corresponding polyhedral complex where this adjacency fails and it must correspond to an $n$-cell of $S(\theta_{i+1})$ that is not contained in the cell of $S(\theta_i)$ corresponding to such adjacency.
\medskip

On the other hand, take a cell $C$ of $S(\theta_{i+1})$ that is not contained in any of the cells of $S(\theta_i)$ and, as we only lifted the polytope $\Delta_i$, the corresponding dual cell on the polyhedral complex has to fail to be adjacent to the same rays.
\end{proof}

\begin{proof}(of the Theorem \ref{tropicalrefinement})
Consider $p$ as a point ($0$-dimensional cell) in the polyhedral complex that is dual to an $n$-cell $D$ of $S(\theta_{i-1})$. Let $v$ be the normal vector to the lifting function $\omega_i: \mathcal{A}_i \xrightarrow[]{} \mathbb{R}$. We have to prove that $v$ lies in an adjacent cell to $0$ in $\mathbb{H}_i$, if and only if, $D$ is contained in a cell $D'$ of $S(\theta_{k})$. \\

Firstly, suppose there was not such cell $D'$. This would mean that the adjacencies would not be preserved and we can find a ray $r$ in $\mathcal{F}_i$ where this property is failing. Consider the ray of a fan $\mathcal{F}_k$ for $k = 0,\dots,i-1$ where this adjacency has changed and this means that we have crossed a hyperplane containing such ray in the previous fan. \\

On the other hand, if there is such cell $D'$, then the lifting of $\Delta_i$ preserves adjacencies. However, if we had moved $v$ to a non-adjacent cell to $0$, we would have crossed a hyperplane therefore, we would be able to find rays in such hyperplane where the adjacencies are not preserved.
\end{proof}
This result extends the proposition 2.11 on \cite[Proposition 2.11]{dandrea2020cannyemiris} and gives a full picture of refinement of mixed subdivisions. Therefore, we naturally understand all the ways to refine a given mixed subdivision $S(\theta_i)$ with affine lifting functions on $\Delta_i$. 
\begin{corollary}
The chambers of the hyperplane arrangement $\mathbb{H}_i$ adjacent to $0 \in \mathbb{R}^n$ are in one to one correspondence to all the possible ways to refine $S(\theta_{i})$. In particular, if $S(\theta_{i})$ is tight, the chambers of $\mathbb{H}_i$ correspond to tight mixed subdivisions.
\end{corollary}

In the context of Theorem \ref{admissiblechain}, in the direction of $v \notin \mathbb{H}_{\Delta}$, the function $\langle\lambda_{i}v, x \rangle$ will reach the hyperplane arrangement $\mathbb{H}_i$ when $\lambda_i = \lambda_{i-1}$. Therefore, for any $0 < \lambda_i < \lambda_{i-1}$, the subdivision $S(\theta_{i+1})$ will refine $S(\theta_{i})$ for $i = 0,\dots,n$. 

\begin{theorem}
    The mixed subdivision $S(\rho)$ in Definition \ref{definitionlifting} is admissible.
\end{theorem}

\begin{proof}

All the lattice points with row content $0$ are $0$-mixed. Therefore, $S(\theta_0)$ satisfies $ii)$ in Definition \ref{admissible}. Let $D$ be an $n$-cell of $S(\theta_i)$. If $\dim D_i = 0$, then the fundamental subfamily of $\mathcal{A}_D$ is at most $\{i\}$ as shown in Remark \ref{remarkresult}. We show that, for our choice of the lifting function, the rest of cells $D$ satisfy $ii)$ in Definition \ref{admissible}. 

\medskip
Let $D \in S(\theta_i)$ such that $\dim D_i > 0$. Suppose that this cell contains a lattice point $b \in \mathcal{B}$ that has row content $i$ but is not $i$-mixed. Therefore, this lattice point $b$ will be in a cell of $S(\rho)$ with a $0$-dimensional $j$-th component for some $j < i$. Take $C \supset D$ in $S(\theta_j)$ containing the previous lattice point $b$. If $\dim C_j > 0$, then the lifting function $\omega_j = \lambda_j\langle v,x\rangle$ takes the same value in all the points of $C_j$. Therefore, the vector $v$ is normal to $C_j$ and has to be contained in the hyperplane arrangement associated to $\Delta$. As this is not the case, $\dim C_j = 0$ and consequently $\dim D_j = 0$, contradicting the initial hypothesis. 
\end{proof}

\cc{Thus, the Canny-Emiris formula holds for the family of lifting functions that we have defined; see \cite[Theorem 4.27]{dandrea2020cannyemiris}.}

\begin{figure}
\centering
\begin{tabular}{||c c c||} 
 \hline
 Step & Lifting & Subdivision \\
 \hline
$S(\theta_0)$ &  \begin{tikzpicture}

\draw[blue] (0,0) -- (3/2,0);
\draw[green] (0,0) -- (3/2,0);
\draw[brown] (0,0) -- (3/2,0);

\filldraw[black] (0,0) circle (1pt) node[anchor=south] {0};
\filldraw[black] (3/2,0) circle (1pt) node[anchor=south] {0};

\end{tikzpicture} & \begin{tikzpicture}

\draw[blue] (0,0) -- (3/2,0);
\draw[green] (0,0) -- (3/2,0);
\draw[brown] (0,0) -- (3/2,0);

\filldraw[black] (0,0) circle (1pt) node[anchor=south] {0};
\filldraw[black] (3/2,0) circle (1pt) node[anchor=south] {$a_{0,0} + a_{1,0} + a_{2,0}$};

\end{tikzpicture} \\ 
 \hline
 $S(\theta_1)$ & \begin{tikzpicture}

\draw[blue] (0,1) -- (1/2,0);
\draw[green] (1/2,0) -- (3/2,0);
\draw[brown] (1/2,0) -- (3/2,0);

\filldraw[black] (0,1) circle (1pt) node[anchor=south] {0};
\filldraw[black] (1/2,0) circle (1pt) node[anchor=east] {$\lambda_0 v_j$};
\filldraw[black] (3/2,0) circle (1pt) node[anchor=south] {$\lambda_0 v_j$};

\end{tikzpicture} & \begin{tikzpicture}

\draw[blue] (0,0) -- (1/2,0);
\draw[green] (1/2,0) -- (3/2,0);
\draw[brown] (1/2,0) -- (3/2,0);

\filldraw[black] (0,0) circle (1pt) node[anchor=south] {0};
\filldraw[black] (1/2,0) circle (1pt) node[anchor=north] {$a_{0,0}$};
\filldraw[black] (3/2,0) circle (1pt) node[anchor=south] {$a_{0,0} + a_{1,0} + a_{2,0}$};

\end{tikzpicture} \\
 \hline
  $S(\theta_2)$ & \begin{tikzpicture}

\draw[blue] (0,3/2) -- (1/2,1/2);
\draw[brown] (1/2,1/2) -- (1,0);
\draw[green] (1,0) -- (3/2,0);

\filldraw[black] (0,3/2) circle (1pt) node[anchor=south] {0};
\filldraw[black] (1/2,1/2) circle (1pt) node[anchor=east] {$\lambda_0 v_j$};
\filldraw[black] (1,0) circle (1pt) node[anchor=north] {$(\lambda_0 + \lambda_1) v_j$};
\filldraw[black] (3/2,0) circle (1pt) node[anchor=south] {$(\lambda_0 + \lambda_1) v_j$};

\end{tikzpicture} & \begin{tikzpicture}

\draw[blue] (0,0) -- (1/2,0);
\draw[brown] (1/2,0) -- (1,0);
\draw[green] (1,0) -- (3/2,0);

\filldraw[black] (0,0) circle (1pt) node[anchor=south] {0};
\filldraw[black] (1/2,0) circle (1pt) node[anchor=north] {$a_{0,0}$};
\filldraw[black] (1,0) circle (1pt) node[anchor=south] {$a_{0,0} + a_{1,0}$};
\filldraw[black] (3/2,0) circle (1pt) node[anchor=west] {$a_{0,0} + a_{1,0} + a_{2,0}$};

\end{tikzpicture} \\
 \hline
  $S(\rho)$ & \begin{tikzpicture}

\draw[blue] (0,7/4) -- (1/2,3/4);
\draw[brown] (1/2,3/4) -- (1,1/4);
\draw[green] (1,1/4) -- (3/2,0);

\filldraw[black] (0,7/4) circle (1pt) node[anchor=north] {0};
\filldraw[black] (1/2,3/4) circle (1pt) node[anchor=east] {$\lambda_0 v_j$};
\filldraw[black] (1,1/4) circle (1pt) node[anchor=west] {$(\lambda_0 + \lambda_1) v_j$};
\filldraw[black] (3/2,0) circle (1pt) node[anchor=north] {$(\lambda_0 + \lambda_1 + \lambda_2) v_j$};

\end{tikzpicture} & \begin{tikzpicture}

\draw[blue] (0,0) -- (1/2,0);
\draw[brown] (1/2,0) -- (1,0);
\draw[green] (1,0) -- (3/2,0);

\filldraw[black] (0,0) circle (1pt) node[anchor=south] {0};
\filldraw[black] (1/2,0) circle (1pt) node[anchor=north] {$a_{0,0}$};
\filldraw[black] (1,0) circle (1pt) node[anchor=south] {$a_{0,0} + a_{1,0}$};
\filldraw[black] (3/2,0) circle (1pt) node[anchor=west] {$a_{0,0} + a_{1,0} + a_{2,0}$};

\end{tikzpicture} \\
 \hline
\end{tabular}
\caption{This table explains how the process of passing from the proposed lifting on $\color{blue} \Delta_0$, $\color{brown} \Delta_1, \color{green} \Delta_2$ to the mixed subdivision works in the $j$-th coordinate for $v_j < 0$ for any of the two components of Example \ref{smallexample}. One clearly sees that, for instance, $\overline{0a_{0,0}e_0} \subset \color{blue} D_0$, if and only if, $x_0 \leq a_{0,0}$ for $x \in \color{blue} D$. The product of two subdivisions of this form gives the mixed subdivision in the figure of Example \ref{smallexample}.}
\label{figure1}
\end{figure}

\section{The case of $n$-zonotopes}
\label{section:3}

\cc{In this section, we restrict our attention to a particular family of Newton polytopes, i.e. $n$-zonotopes. In this subfamily, it will be easier to give a combinatorial description of the lattice points associated to the rows of the Canny-Emiris matrix, after applying the greedy algorithm.} For simplicity, we suppose that the lattice is $M = \mathbb{Z}^n$. 

\begin{definition}
A zonotope is a polytope given as a sum of line segments. An $n$-zonotope is generated by $n$ line segments, which span an $n$-dimensional lattice.
\end{definition}

Consider linearly independent vectors $v_1,\dots,v_n \in \mathbb{Z}^{n}$ and the line segments $\overline{0v_1},\dots,\overline{0v_n} \subset \mathbb{R}^n$ forming an $n$-zonotope $Z \subset \mathbb{R}^n$. If the Newton polytopes are $n$-zonotopes whose defining line segments are integer multiples of the $\overline{0v_j}$, we can write the supports of the system as:
\[ 
\mathcal{A}'_i = \big\{ \sum_{j = 1}^n \lambda_j v_j \in \mathbb{Z}^n\, | \quad \lambda_j \in \mathbb{Z}, \quad 0  \leq \lambda_j \leq a_{ij}\big\}.
\]
for some $a_{ij} \in \mathbb{Z}_{>0}$. Let $V$ be the nonsingular matrix whose columns are the coordinates of the $v_j$ in the canonical basis of $\mathbb{Z}^n$ for $j = 1,\dots,n$ and consider it as a monomorphism of lattices $V: \mathbb{Z}^n \xrightarrow[]{} \mathbb{Z}^n$ of rank $n$. Let $e_1,\dots,e_n$ be the canonical basis of $\mathbb{Z}^n$.

\begin{corollary} \label{corollaryresultant}
Let $\mathcal{A}'_0, \dots, \mathcal{A}'_n$ be the previous family of supports, then $\Res_{\mathcal{A}'} = \Res_{\mathcal{A}}^{|\det(V)|}$, where:
\[ 
\mathcal{A}_i = \big\{(b_j)_{j = 1,\dots,n} \in \mathbb{Z}^n \quad | \quad 0 \leq b_j \leq a_{ij} \big\} \quad i = 0,\dots,n
\]
\end{corollary}

\begin{proof}
Using Lemma \ref{lemmamonomrphism}, we can view the map $V: \mathbb{Z}^n \xrightarrow[]{} \mathbb{Z}^n$ as a monomorphism of lattices sending the canonical basis $e_i$ to $v_i$ for $i = 1,\dots,n$. The absolute value of the determinant $|\det(V)|$ is the index of the image. This last result follows from the reduction of $V$ to its Smith normal form \cite[Theorem 2.3]{stanleysnf}.
\end{proof}

\begin{remark}
\label{normals}
The normal vectors of the $n$-zonotope are given by $n$ pairs $(\eta_j,-\eta_j)_{j=1,\dots,n}$ in $N$. The results that follow in this section could be proved without using Corollary \ref{corollaryresultant}, after changing $b_j$ by $\langle b, \eta_j\rangle$ for $j = 1,\dots,n$, choosing $\eta_j$ to be the element in the pair such that $0 \leq \langle b,\eta_j \rangle \leq a_{ij}$.

\end{remark}

In order to prove our results, we assume that the $a_{ij}$ are ordered, meaning that $0 < a_{0j} \leq a_{1j} \leq \dots \leq a_{n-1j}$ and $j = 1,\dots,n$, where we exclude $\mathcal{A}_n$ form this assumption; notice that Example \ref{exampledrop} wouldn't satisfy this property without excluding $\mathcal{A}_n$. Consider a translation $\delta \in \mathbb{R}^n$ which is negative in each component and small enough. Then, the lattice points in translated cells of a mixed subdivision of the previous system are:
\[
\mathcal{B} = \big\{(b_j)_{j = 1,\dots,n} \in \mathbb{Z}^n \quad | \quad 0 \leq b_j < \sum_{i = 0}^n a_{ij}\big\}.
\]

Let $v \notin \cup_{i = 1}^n\{x_j = 0\}$ define the mixed subdivision $S(\rho)$ as in the previous section. We assume $v_j < 0$ for $j = 1,\dots,n$ and get the following result.
\begin{proposition}\cite[Proposition 3.1]{issacversion}
\label{mixedsubdivision}
Let $b \in \mathcal{B}$ and $i \in \{0,\dots,n\}$. Then:
\[ t_{b,i} = \big|\big\{j \in \{1,\dots,n\} \quad | \quad \sum_{k = 0}^{i-1} a_{kj} \leq  b_j  < \sum_{k = 0}^i a_{kj}\big\}\big|\]
and the row content $i(b)$ is the maximum index in $\{0,\dots,n\}$ such that:
\[
\not\exists j \in \{1,\dots,n\}: \quad \sum_{k = 0}^{i(b)-1} a_{kj} \leq  b_j  < \sum_{k = 0}^{i(b)} a_{kj}
\]
with the support $a(b) \in \mathcal{A}_{i(b)}$ satisfying:
\[a(b)_j = \begin{cases} 
      0   & b_j < \sum_{k = 0}^{i(b)-1}a_{kj}, \\
      a_{i(b)j} & b_j \geq \sum_{k = 0}^{i(b)}a_{kj}.
   \end{cases}\]
\end{proposition}
\begin{remark}
If $v_j > 0$, we would change the inequalities by $\sum_{k = i}^{n} a_{kj} \leq  b_j  < \sum_{k = i - 1}^i a_{kj}$, but the results that follow would not change. Any other mixed subdivisions of this particular system can also be formed this way.
\end{remark}
\begin{definition}
\label{definitiontypefunction}
The type function $\varphi_b: \{1,\dots,n\} \xrightarrow{} \{0,\dots,n\}$ associated to each lattice point $b \in \mathcal{B}$ is defined as the vector of indices satisfying:
\[\sum_{k = 0}^{\varphi_b(j)-1}a_{kj} \leq b_j < \sum_{k = 0}^{\varphi_b(j)}a_{kj}\]
Following Proposition \ref{mixedsubdivision}, it satisfies that $t_{b,i} = |\varphi_b^{-1}(i)|$.
\end{definition}




\begin{definition}
\label{greedysubset}

We define the greedy subset $\mathcal{G} \subset \mathcal{B}$ to be formed by all the lattice points $b \in \mathcal{B}$ such that:
\[\sum_{i = 0}^I t_{b,i} \leq I + 1 \quad \forall I < n.\]
\end{definition}





    
\begin{theorem}\label{maintheorem}\cite[Theorem 3.1, Theorem 3.2]{issacversion}  Let $b \in \mathcal{G}$ and $b' \notin \mathcal{G}$. Then, $b' \notin b - a(b) + \mathcal{A}_{i(b)}$. Moreover, if $b,b' \in \mathcal{G}$ and $b$ lines in a mixed cell, we can reach $b'$ from $b$.
\end{theorem}

\begin{corollary}
\label{bigconclusion}\cite[Corollary 3.3]{issacversion}
The size of the matrix $\mathcal{H}_{\mathcal{G}}$ is:
\[ \sum_{\varphi_b: \{1,\dots,n\} \xrightarrow[]{} \{0,\dots,n\}}\prod_{j = 1}^na_{\varphi_b(j)j}\]
where the sum is over the functions that satisfy $\varphi_b^{-1}(\{0,\dots,I\}) \leq I + 1 \quad \forall I < n$.
\end{corollary}

\color{black}
In \cite{issacversion}, the authors showed that if one is able to embed the mixed subdivision of a given polynomial system into the mixed subdivision of an $n$-zonotope system, this construction also applies there. This is the case of multihomogeneous systems, in which the Newton polytopes are products of simplices. For such a case, there exist exact determinantal resultant formulas obtained by using the Weyman complex and other tecniques \cite{bender2021koszultype,mixedgrobnerbasistowards,dickensteinemirismultihomo03,emimantzdeterminant,sturmfelszelevinsky}. \cc{However, none of those formulas is directly using the Canny-Emiris formula. The use of type functions makes our approach easier to generalize to general Newton polytoeps. In the next section, we state a conjecture on which is the expected minimal size of these matrices.}

\begin{figure}
\centering
\begin{tabular}{ |c|c|c|c| } 
 \hline
 Dimension & Canny-Emiris & Greedy & Resultant degree \\ 
 \hline
 2 & 9 & 8 & 6 \\ 
 3 & 64 & 50 & 24  \\ 
 4 & 625 & 432 & 360 \\ 
 5 & 7776 & 4802 & 3720 \\ 
 \hline
\end{tabular}
\caption{This table represents the size of the matrices we achieve for zonotopes of dimensions from $2$ to $5$ with $a_{ij} = 1$ using the greedy approach versus the original Canny-Emiris formula. We also compare to the degree of the resultant.}
\label{figure2}
\end{figure}

\section{Which are the minimal matrices?}
\label{sec::4}

A very natural question for sparse polynomial systems is which are the matrices of smallest size that one can build to represent the sparse resultant. In this section, we explain how this is related to the Hilbert function of a generic polynomial system with some given supports and state a conjecture saying that if the Hilbert function of a generic system attains the value zero at some multi-degree $\mathbf{m}$, there is a lifting and a greedy subset $\mathcal{G}$ associated with the degree $\mathbf{m}$ providing a resultant formula. We state the conjecture for multi-homogeneous systems but one can easily generalize it to the sparse case by using degrees in the Cox ring of a toric variety \cite{coxlittleschneck}.
\medskip

One can motivate this question through Gr\"obner basis. For homogeneous systems, it is known that the computation of a Gr\"obner basis finishes once a certain degree in the generators is achieved. This degree coincides with an algebraic invariant called Castelnuovo-Mumford regularity and it was proven by Bayer and Stillman in \cite{bayer_criterion_1987} that it provides a tight bound to the degrees involved in a Gr\"obner basis. This result \cc{uses the degree reverse lexicographical order after a generic change of coordinates}. 
\medskip

In the case of generic coefficients (still in the homogeneous case), this coincides with the degree at which one has to arrive in order to build resultant matrices: the Macaulay bound \cite{emirismourrain,macaulay}. There are analogues of the Macaulay bound \cite[Proposition 8.2.2]{benderthesis} but they \cc{are} not tight in general \cite{Awane2005}. In terms of the resultant construction, we can relate this to the mixed subdivisions that we considered in the previous sections.
\medskip

In \cite{issacversion}, the authors only considered affine lifting functions, for the sake of simplicity on the combinatorics of the greedy algorithm. However, the results are known to be not optimal, in the sense that there exist other lifting functions that provide smaller resultant matrices. Therefore, a natural question is to ask which subsets $\mathcal{G} \subset \mathcal{B}$ can be obtained using the greedy algorithm for some lifting function and which of them are minimal. 

\begin{example}
\label{exampledrl}
Consider the same bilinear system as in Example \ref{smallexample}. Another possible non-affine mixed subdivision $S(\rho)$ is the following:
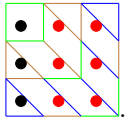
\[\begin{tikzpicture}

\draw[brown] (0,1/2) -- (1/2,1/2);
\draw[brown] (0,1/2) -- (0,1);
\draw[brown] (0,1) -- (1/2,1/2);
\draw[brown] (1/2,0) -- (1,0);
\draw[brown] (1/2,1) -- (1,1/2);
\draw[brown] (1/2,3/2) -- (1,1);
\draw[brown] (1/2,3/2) -- (1,3/2);
\draw[brown] (1,1) -- (1,3/2);
\draw[brown] (3/2,1/2) -- (3/2,1);

\draw[blue] (0,0) -- (1/2,0);
\draw[blue] (0,0) -- (0,1/2);
\draw[blue] (1/2,0) -- (0,1/2);
\draw[blue] (1,0) -- (1/2,1/2);
\draw[blue] (3/2,0) -- (1,1/2);
\draw[blue] (3/2,1/2) -- (1,1);
\draw[blue] (3/2,3/2) -- (3/2,1);
\draw[blue] (3/2,3/2) -- (1,3/2);
\draw[blue] (1,3/2) -- (3/2,1);

\draw[green] (0,3/2) -- (0,1);
\draw[green] (1/2,1) -- (0,1);
\draw[green] (1/2,1) -- (1/2,3/2);
\draw[green] (0,3/2) -- (1/2,3/2);
\draw[green] (1,1/2) -- (1/2,1/2);
\draw[green] (3/2,0) -- (1,0);
\draw[green] (1,1) -- (1,1/2);
\draw[green] (3/2,0) -- (3/2,1/2);

\filldraw[black] (0+1/5,1/5) circle (2pt) node[anchor=east] {};
\filldraw[black] (0+1/5,1/2+1/5) circle (2pt) node[anchor=east] {};
\filldraw[black] (0+1/5,1+1/5) circle (2pt) node[anchor=east] {};
\filldraw[red] (1/2+1/5,0+1/5) circle (2pt) node[anchor=south ] {};
\filldraw[red] (1/2+1/5,1/2+1/5) circle (2pt) node[anchor= south] {};
\filldraw[red] (1/2+1/5,1+1/5) circle (2pt) node[anchor=west] {};
\filldraw[red] (1+1/5,0+1/5) circle (2pt) node[anchor=south] {};
\filldraw[red] (1+1/5,1/2+1/5) circle (2pt) node[anchor=west] {};
\filldraw[red] (1+1/5,1+1/5) circle (2pt) node[anchor=west] {};

\end{tikzpicture}.
\]
The red dots indicate the greedy subset that one obtains by starting the algorithm at the lattice points in mixed cells. A possible lifting function giving this mixed subdivision is $\omega_0 = (0,1,1,3), \omega_1 = (0,2,2,5), \omega_2 = (0,3,3,7)$, which is not affine.
\end{example}

Assume that we are working with coefficients in the field of complex numbers $\mathbb{C}$. Consider $\mathcal{A}_0,\dots,\mathcal{A}_n$ be a family of supports corresponding to a multihomogeneous system. \cc{Assume also that }each of the $\mathcal{A}_i$ can be associated to a multidegree in $\mathbf{d}_i \in \mathbb{Z}^d$. The generic Hilbert function is defined as:
$$ \HF(\mathbf{d}) = \dim(S/I)_{\mathbf{d}} \quad \mathbf{d} \in \mathbb{Z}^d$$
where $I$ is the ideal in $\mathbb{C}[M]$ after specializing the $u_{i,a}$ to generic values in $\mathbb{C}$. This generic Hilbert function exists as we cannot have two different generic behaviours for coefficients in $\mathbb{C}$. Moreover, it is natural to think that we can associate some of the subsets $\mathcal{G} \subset \mathcal{B}$ to some multi-degrees. For instance, the whole subset $\mathcal{B}$ can be naturally associated to the multi-homogeneous Macaulay bound. This idea can easily be extended to the sparse case by considering generic values of the Hilbert function associated to degrees in the Cox ring of a toric variety.

\begin{conjecture}
\label{conjecture}
Assume that $\mathcal{G} \subset \mathcal{B}$ is a set of lattice points in the translaated cells that corresponds to a multi-degree $\mathbf{d} \in \mathbb{Z}^d$. If $\HF(\mathbf{d}) = 0$, then there is a lifting function $\omega \in \prod_{i = 0}^n\mathbb{R}^{\mathcal{A}_i}$ such that $\mathcal{G}$ is the greedy subset of such system. Moreover, if $\HF(\mathbf{d}') \neq 0$ for $\mathbf{d}' \lneq \mathbf{d}$, $\mathcal{G}$ contains no greedy subset.
\end{conjecture}

It is natural to think that this lifting function must be related to the degree reverse lexicograpical order. Namely, that for two monomials $x^A, x^B \in k[M]$, we have:
$$ x^A <_{\mathfrak{drl}} x^B \iff \omega(A) \leq \omega(B). $$
Here $\omega(A)$ refers to evaluating the exponents of $x^A$ in the inf-convolution of $\omega$ as in Definition \ref{infconvolution}.

\begin{example}
The lifting function given in Example \ref{exampledrl} satisfies this degree reverse lexicographical condition. Moreover, the subset $\mathcal{B}$ obtained by considering all the lattice points in translated cells can be related to the bi-degree $(2,2)$. However, the greedy subset $\mathcal{G}$ that we have found corresponds to the bi-degree $(2,1)$. In particular, this bi-degree corresponds to some existing exact resultant formulas \cite{dickensteinemirismultihomo03}.
\end{example}

\section{Applications}
\label{sec::5}

Resultants and matrix representations of polynomial systems can be used in applications for solving polynomial systems. The ideas that are useful for solving a large number of problems in applications and offer standard methods that can be applied in several contexts. We add two examples of these applications in computer vision and design: the $5$-points problem and the implicitization problem. We have updated the \href{https://github.com/carleschecanualart/CannyEmiris}{\color{blue}JULIA implementation \color{gray}} in \cite{issacversion} so that it includes these examples.

\subsection{Computer vision: the $5$ point problem}

A typical computer vision problem is interested in computing the displacement of a camera between two positions in a static environment. Namely, we would like to find the displacement of a rigid body between two snapshots taken by a stationary camera. The identifiable features of the body include only points. 
\medskip

Usually, a minimum number of 5 point matches is available. The algebraic problem reduces to a well-constrained system of polynomial equations and we are able to give a closed-form solution.
Typically, computer vision applications use at least 8 points in order to reduce the number of possible solutions to one, in generic coordinates. In addition, computing the displacement reduces to a linear problem and the effects of noise in the input can be diminished \cite{LonguetHiggins1981ACA}. 
\medskip

Let $a_i \in (\mathbb{R}^3)$ for $i = 1,\dots,5$ be the $5$ points in the first snapshot and $a'_i \in (\mathbb{R}^3)$ for $i = 1,\dots,5$ be the points in the second snapshot. A quaternion formulation of this problem was proposed in \cite{hornreconstruction}. This quaternion formulation reduces the problem to solving the polynomial system given by the following equations in the variables $q \in \mathbb{R}^3$ (representing a rotation) and $d \in \mathbb{R}^3$ (representing a translation)
$$
(a_i^Tq)(d^Ta'_i) +  a_i^Ta'_i + (a_i \times q)^Ta'_i + (a_i \times q)^T(d \times a'_i) + a_i^T(d \times a'_i) = 0, \quad i = 1,\dots,5,
$$
$$
1 - d^Tq = 0,
$$
where $\times$ represents the usual exterior product. The first five equations represents each of the $5$ displacements while the last one represents a normalization between the vectors $q,d$. This system is bilinear in the two groups of variables. We can solve it by building the $u$-resultant. Namely, we introduce a new linear equation $P_u = u_0 + u_1d_1 + u_2d_2 + u_3d_3 + u_4q_1 + u_5q_2 + u_6q_3$.
Once we consider the resultant of this system, we get a polynomial that factors into linear forms, whose coefficients are the values of the solutions (they are a finite number in this case); see also \cite{emiristhesis} for a similar approach.
\medskip

Using our approach, we get a square matrix of dimension $784$ while the degree of the resultant is $770$.

\subsection{An implicitization problem}

The implicitization problem has many appications in computer-aided design, since it is one of the major problems in changing representation. It has been addressed by a multitude of methods, including sparse resultants e.g. \cite{kalinkaetal}. Namely, given the equations of a surface in $\mathbb{R}^2$ given by two parameters $s,t$:
$$ 
\phi: \mathbb{R}^2 \xrightarrow[]{} \mathcal{X} \subset \mathbb{R}^3, \quad (s,t) \xrightarrow[]{} (\phi_1(s,t), \phi_2(s,t), \phi_3(s,t)),
$$
we would like to find a polynomial equation in three variables $P(X,Y,Z)$ that represents the surface $\mathcal{X} = \{P(X,Y,Z) = 0\}$. Therefore, we are forced to eliminate the variables $s,t$ from the polynomial system:
$$ X - \phi_1(s,t), Y - \phi_2(s,t), Z - \phi_3(s,t).$$
This sums up to computing the resultant, assuming that the Newton polytope of $\phi_i$ for $i = 1,\dots,3$ is a zonotope.

\begin{example}
    Suppose that $\phi_1 = 1 + 3s - 2st + s^2t, \phi_2 = -1 -3s + 4st + 5s^2t, \phi_3 = -2 + 5s + 4st - 1s^2t$. Our construction provides the following matrix for the resultant computation of the implicitization problem:
  $$  \begin{pmatrix}
Z + 2 & -4 & 0 & -5 & 1 & 0 & 0 & 0 \\
Y + 1 & -4 & 0 & 3 & -5 & 0 & 0 & 0 \\
0 & 0 & Z + 2 & -4 & 0 & -5 & 1 & 0 \\
0 & 0 & 0 & \color{green} Z + 2 & -4 & 0 & -5 & \color{green} 1 \\
X - 1 & 2 & 0 & -3 & -1 & 0 & 0 & 0 \\
0 & 0 & Y + 1 & -4 & 0 & 3 & -5 & 0 \\
0 & 0 & X - 1 & 2 & 0 & -3 & -1 & 0 \\
0 & 0 & 0 &\color{green} Y + 1 & -4 & 0 & 3 & \color{green} -5
    \end{pmatrix} $$
    where the principal minor is marked in green.
\end{example}


\textbf{Acknowledgements.}
Both authors have received partial support from EU’s H2020 research \& innovation programme under the Marie Skłodowska-Curie grant agreement No 860843 (\href{http://grapes-network.eu/}{GRAPES}). We thank Elias Tsigaridas and Christos Konaxis for their support, and Carlos D'Andrea for the time spent answering our questions. We thank the anonymous reviewers of the first version of this article published in the proceedings of ISSAC 2022 for relevant comments.

\printbibliography[]

\appendix
\color{black}

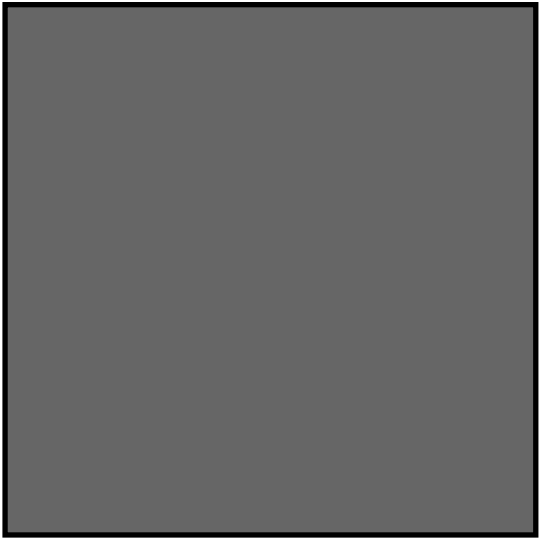
\end{document}